\documentclass[11pt]{article}
\usepackage[american]{babel}

\usepackage{graphicx}
\usepackage{comment}
\usepackage{color}
\usepackage{algorithm}
\usepackage{algorithmic}
\usepackage[left=2cm,right=2cm]{geometry}
\usepackage{amssymb,amsfonts}
\usepackage{amsthm,amsmath}
\usepackage{mathtools}
\usepackage[hidelinks]{hyperref}
\newcommand {\mat}[1]{\left[\begin{array}{#1}}
\newcommand {\rix}{\end{array}\right]}

\newtheorem{theorem}{Theorem}
 \newtheorem{corollary}[theorem]{Corollary}
\newtheorem{lemma}[theorem]{Lemma}
 
 \newtheorem{definition}[theorem]{Definition}
\newtheorem{remark}[theorem]{Remark}
\newtheorem{example}[theorem]{Example}
\newcommand{\diag}     {\mathop{\rm diag}\nolimits}
\newcommand{\rank}     {\mathop{\rm rank}\nolimits}
\newcommand{\R}{\mathbb R}
\newcommand{\C}{\mathbb C}
\def\eproof{\space
        {\ \vbox{\hrule\hbox{\vrule height1.3ex\hskip0.8ex\vrule}\hrule}}
        \bigskip}
\newcommand{\nnrm}[1]{{\left\vert\kern-0.25ex\left\vert\kern-0.25ex\left\vert #1
		\right\vert\kern-0.25ex\right\vert\kern-0.25ex\right\vert}}

\title{Characterization of stability radii for robustly asymptotically  stable dissipative Hamiltonian differential-algebraic systems}

\author{Peter Benner, Volker Mehrmann, Anshul Prajapati, and Punit Sharma}

\begin{document}
\maketitle

\begin{abstract}
    We study aspects of robust stability of linear time-invariant dissipative Hamiltonian differential-algebraic systems.  
    In particular, we provide a complete characterization when such systems are robustly asymptotically stable, i.e., the following three properties are satisfied: regularity, all finite eigenvalues are in the left half complex plane, and eigenvalues at $\infty $ are semisimple.
    These criteria can be checked numerically. We derive exact conditions and bounds when robust asymptotic stability is lost under structure-preserving perturbation. The results are illustrated by numerical examples.
	\end{abstract}


Msc Classification: {93D20, 93D09, 65F15}

\section{Introduction}%
\label{sec:intro}

In this paper, we analyze the robust asymptotic stability under structure-preserving perturbations for the class of \emph{linear time-invariant dissipative Hamiltonian differential-algebraic equation (dHDAE)} systems of the form
\begin{equation}\label{dhdae}
    E\dot{x}=(J-R)Qx ,
\end{equation}
with coefficient matrices $E, J, R, Q\in \mathbb C^{n,n}$ (the  set of $n\times n$ complex matrices), where the coefficients satisfy
\begin{equation}\label{dhprop}
  E^*Q=Q^*E\geq 0,\quad Q^*JQ=-Q^*J^*Q,\quad Q^*RQ= Q^*R^*Q\geq 0.
\end{equation}
Here, $Q^*$ denotes the conjugate transpose of a matrix $Q$ and $R\, (\geq)> 0$ denotes that $R$ is positive (semi-) definite.
The function $\mathcal H: x\mapsto \frac 12 x^*E^*Qx$, which describes the energy stored in the system, is called the \emph{Hamiltonian} of the system.

Such dHDAE systems arise, e.g. in  structural mechanics and fluid dynamics; see \cite{GudLMS22} for a large list of applications. To illustrate the need for a detailed structured perturbation analysis, we present the following four examples.
\begin{example}\label{ex:MDK}{\rm
The modified nodal analysis is a classical technique in circuit simulation \cite{EstT00,GueF99a,GueF99b,KunM24}. The state vector of unknowns
consists of all node potentials~$u$ and all branch currents~$i_L,i_V$
of current-controlled elements. After linearization,  the system locally has the form of a dHDAE
$
 E
\dot x = (J-R) x+  B w,
$
given by
\[
\begin{bmatrix}
A_C C A_C^\top & 0 & 0 \\
0 & L & 0 \\
0 & 0 & 0
\end{bmatrix}
\begin{bmatrix}
\dot u \\ \dot i_L \\ \dot i_V
\end{bmatrix}
=
-\begin{bmatrix}
A_R R A_R^\top & A_L & A_V \\
-A_L^\top &0  & 0 \\
-A_V^\top & 0 & 0
\end{bmatrix}
\begin{bmatrix}
u \\ i_L \\ i_V
\end{bmatrix}\\
\ +
\begin{bmatrix}
-A_I &0\\ 0 & 0\\ 0 & -I
\end{bmatrix}\begin{bmatrix}
i_{in} \\ v_{in}
\end{bmatrix},\nonumber
\]
with input voltage $V_{in}$ and input currents $i_{in}$, and
where the  device matrices $C$, $L$, and $R$ associated with capacitors, inductors and resistors are typically  symmetric and positive-definite.
This representation is subject to linearization and modeling errors in the matrices $C,L,R$, while typically the $A_X$ matrices, with $X\in\{C,R,L,V,I\}$, are fixed, except if the circuit structure is changed.
}
\end{example}

\begin{example}\label{ex:MDK1}{\rm
The finite element modeling of the acoustic field in the interior of a car leads after some simplifications, see
Example~13 from~\cite{BeaMXZ18}, to a  differential–algebraic system of the form
    \[
    M \ddot{p} + D \dot{p} + K p = B_1 u,
    \]
    where $p$ denotes the vector of coefficients corresponding to the acoustic pressure in the air and the structural displacements. The term $B_1 u$ represents an external force. The matrix $M$ is a symmetric positive semidefinite mass matrix, $D$ is a symmetric positive semidefinite damping matrix, and $K$ is a symmetric positive definite stiffness matrix. The resulting first-order formulation yields the state equation of a pHDAE system, $E \dot{x} = (J-R) x + B u$, where the state vector is $x = \mat{c}\dot{p}\\ p \rix $ and the matrices are given by
    \[
    E = \mat{cc} M & 0 \\ 0 & K \rix,
    \quad
    J = \mat{cc} 0 & -K \\ K & 0 \rix,
    \quad
    R = \mat{cc} D & 0 \\ 0 & 0 \rix, \text{ and}
    \quad
    B = \mat{c} B_1 \\ 0 \rix.
    \]
    The model is subject to modeling and simplification errors in the coefficients.
}
\end{example}
\begin{example}\label{ex:stokes}
{\rm
The space discretization of the unsteady (in)compressible Stokes or linearized Navier-Stokes equations via finite element or finite difference methods typically leads to dHDAE systems of the form
\begin{equation}\label{P-instat-op-general}
\begin{bmatrix} M_v & 0 \\ 0 & M_p \end{bmatrix}
 \begin{bmatrix} \dot{v} \\ \dot{p} \end{bmatrix} =
\begin{bmatrix} A & B\\ -B^* & C\end{bmatrix}
\begin{bmatrix} v\\p\end{bmatrix},
\end{equation}
where $M_v=M^*_v>0$ is the velocity mass matrix, $M_p=M^*_p\geq 0$ is the pressure mass matrix which is zero  in the case of an  incompressible fluid and very close to zero if the flow is close to being incompressible,
$B$ is the discretized gradient operator (normalized so that it is of full column rank),  $C=C^*\geq 0$ is an often employed stabilization term in numerical analysis that is typically positive definite but of small norm or zero if no stabilization is used, and  $A$ is a diffusion and dissipation matrix that has a positive semidefinite Hermitian part; see, e.g., \cite{EmmM13}. Here, $v$ and $p$ denote the discretized velocity and pressure, respectively. All coefficient matrices are subject to discretization and modeling errors.}
\end{example}	

\begin{example}\label{ex:poro}{\rm
The discretization of the equations that model the deformation of porous media saturated by an incompressible viscous fluid in first-order formulation as in~\cite[Section~3.4]{AltMU21} leads to a dHDAE of the form
\begin{align}
\label{eqn:twoField:opMatrix3}
\begin{bmatrix}Y & 0 & 0\\ 0 & A & 0 \\ 0 & 0 & M\end{bmatrix}
 \begin{bmatrix} \dot{w} \\\dot{u}\\ \dot{p} \end{bmatrix}
  = \begin{bmatrix} 0 & -A & D^* \\ A & 0 & 0 \\ -D & 0 & -K \end{bmatrix}
 \begin{bmatrix} w\\ u\\ p \end{bmatrix} ,
\end{align}

where $A,M,Y$ are Hermitian positive definite (with $Y$ being  of very small norm), $K$ is typically Hermitian positive semidefinite, and $D$ is general, non-Hermitian. Here, $u$ represents the discretized displacement field, $w$ the associated discretized velocities, and $p$ the discretized pressure.
A common simplification of the model is to set $Y=0$, then $E={\rm diag}(Y,A,M)=E^*\geq 0$ becomes singular.
All coefficient matrices are subject to discretization and modeling errors.
}
\end{example}

The class of dHDAE systems posesses numerous significant geometric and algebraic properties that are nicely encoded in their representation; see~\cite{AchAM21,DalV98,JacZ12,OrtVME02,MehU23}.
In this paper, we focus on the robust asymptotic stability of dHDAE systems. Before we do so, we first recall the definition of robust asymptotic stability for general unstructured linear DAEs, see \cite{ByeN93,DuLM13}.
\begin{definition}\label{def:robasyst}
A general homogeneous linear DAE $E\dot x=Ax$ with constant coefficients
$E,A\in \mathbb C^{n,n}$ is called
\emph{robustly asymptotically stable} if the pair $(E,A)$ is regular, i.e., $\det (\lambda E-A)\not \equiv 0$,  of index at most one, i.e., all blocks in its Kronecker canonical form (see Theorem~\ref{th:kcf} below) associated with the eigenvalue $\infty$ have size at most one,  and has all its finite eigenvalues in the open left half of the complex plane.
\end{definition}
Note that for the analysis of robust asymptotic stability, only the homogeneous part of the system is relevant, so the definition also applies to nonhomogeneous systems.

The following example, which is  modified from an example in \cite{DuLM13}, provides an illustration of the possible difficulties
in the robust asymptotic stability of DAEs under small perturbations.
\begin{example}\label{ex:perind}
{\rm
Consider the DAE (it is actually a   dHDAE)
\begin{equation}\label{exp1}
\mat {cc} 1 & 0 \\ 0  & 0 \rix \mat {c} \dot x_1 \\ \dot x_2 \rix =\mat {cc} 0 & 1 \\ -1 & 0\rix \mat {c} x_1 \\x_2 \rix,
\end{equation}
which has only the trivial solution $x_1=x_2=0$.

If we perturb (\ref{exp1}) by a small $\varepsilon$ to
\begin{equation}\label{exp1p}
\mat {cc} 1 & 0 \\ 0  & 0 \rix \mat {c} {\dot x}_1 \\ {\dot x}_2 \rix =\mat {cc} 0 & 1 \\ -1 & \varepsilon\rix \mat {c} x_1 \\x_2 \rix,
\end{equation}
then solving the second equation of (\ref{exp1p}) for $x_2$
and substituting into the first equation, we obtain
\begin{equation}\label{uodeexp1}
{\dot x}_1=(1/\varepsilon)x_1.
\end{equation}
Clearly, if $\varepsilon>0$, then the perturbed system (\ref{exp1p})  is not a dHDAE any longer, since the Hermitian part of the right hand side $\mat{cc} 0 & 0 \\ 0 &-\epsilon \rix\not \geq 0$. But if $\varepsilon <0$, then the solution  $x_1$ is asymptotically stable, but it qualitatively differs from the solution of the original system (\ref{exp1}). For an arbitrarily prescribed initial value $x_1(0)\neq 0$, the initial value problem for (\ref{uodeexp1}) has a unique solution. Furthermore, the value of $x_2(0)$ is not required and is uniquely determined by $x_1(0)$.
In fact, this small perturbation has changed the index of the dHDAE (\ref{exp1}).
}
\end{example}

As the presented examples demonstrate, for the analysis of models with uncertain data it is important to know whether a dHDAE system is robustly asymptotically stable. Further applications are presented e.g. in~\cite{GraMQSV16,Mar86,MarL90,MarPR07,RomM09,Sch90}. In this context, to quantify which perturbations and errors can be tolerated, it is important to determine the distance to the nearest system that is not robustly asymptotically stable.

If the product in \eqref{dhdae} is multiplied out to create a matrix $A=(J-R)Q$ and the dHDAE structure  is ignored, then    it is very difficult to analyze  robust asymptotic stability under perturbations, see \cite{ByeHM98,GugL25,GugLM17} and \cite{DuLM13} for a detailed analysis.

We will show that for dHDAE systems under structure-preserving perturbations, the asymptotic stability, regularity, and index of the system becomes evident from the structure of the coefficients and the perturbation analysis is possible. In particular, we present a characterization of  the nearest dHDAE system that is singular, of high index, or has purely imaginary eigenvalues.

For complex unstructured linear systems $\dot x=Ax$ that are asymptotically stable,
the computation of the \emph{stability radius}, i.e., the smallest perturbation that moves an eigenvalue to the imaginary axis, is a very difficult problem, see, e.g., \cite{Bye88}. This is even more  the case for unstructured DAE systems, see \cite{BenV12a,ByeN93,DuLM13}.
The situation is  much better for  dHDAE systems. It is well-known \cite{MehMW18} that the finite eigenvalues of the matrix pencil  $\lambda E-(J-R)Q$ associated with \eqref{dhdae} are in the closed left half complex plane, and the purely imaginary eigenvalues (except possibly the eigenvalue $0$) are semisimple. Furthermore, the index of the pencil is at most two and the singular part (if it exists) is easily characterized.
In general, however, dHDAE systems are  not necessarily robustly asymptotically stable.

The computation of the smallest perturbation that destroys robust asymptotic stability has recently been studied for ordinary dissipative Hamiltonian systems  (the case that $E=I$) under structure-preserving perturbations; see~\cite{MehMS17,MehMW21}. In~\cite{MehMS16}, bounds for the smallest perturbations were examined by analyzing structure-preserving perturbations to $J$, $R$, and $Q$ individually, resulting in computable formulas. In~\cite{AliMM20}, stability radii approximations for large-scale port-Hamiltonian systems were derived, focusing solely on perturbations in $R$. Additionally, in~\cite{BagGS21}, perturbations to $J$ and $R$ were considered, and a lower bound was derived for the structured stability radii.
In \cite{GilMS18}, the problem  of finding a smallest perturbation that moves a dHDAE system that is not robustly asymptotically stable to the boundary of the set of  robustly asymptotically stable dHDAE systems has been studied.
For the case of dissipative Hamiltonian ODE systems, i.e.\ $E=I$, the distance to instability has recently  been characterized in \cite{BenMPS25}.

In this paper, we study the opposite direction, i.e., we determine formulas and bounds for the distance to the boundary
of the set of robustly asymptotically  stable  dHDAE systems under structure-preserving perturbations.
The paper is organized as follows.
In Section~\ref{sec:prelim}, we present the notation and some preliminary results. Moreover, we  recall some basic properties of dHDAE systems, and we also introduce the concept of robust asymptotic stability.
In Section~\ref{sec:dist}, we present explicit formulas for the structured distance to singularity, the structured distance to the system with index greater than one, and the structured distance to the nearest system with purely imaginary eigenvalues. These results are derived for structured perturbations in all the coefficients. In Section~\ref{sec:partial}, we also present perturbation results for the special cases in which only some of the coefficients are perturbed. Numerical examples are presented in Section \ref{sec:numerical}. We summarize the results and give an outlook on further research problems in Section~\ref{sec:conclusion}.

\section{Notation and preliminaries}\label{sec:prelim}
In the following, $\|\cdot\|$ denotes the spectral norm of a vector or a matrix. By $\Lambda(A)$ we denote the spectrum of a matrix $A \in \C^{n,n}$. 
We use the notation $A \geq 0$ and $A \leq 0$ if $A \in \C^{n,n}$ is Hermitian and positive or  negative semidefinite, respectively, and $A > 0$ if $A$ is Hermitian positive definite.
 For a complex number $z$, $\Im(z)$ and $\Re(z)$, respectively, denote the imaginary and the real part of $z$. We denote by $\sigma_{\text{min}}(A)$  the smallest singular value of a matrix $A$. If $R$ is Hermitian, then $\lambda_{\text{max}}(R)$ and $\lambda_{\text{min}}(R)$ denote its largest and smallest eigenvalue, respectively.
We denote the identity matrix of size $n$ by $I_n$, leaving off the index if the dimension is clear from the context.

We will frequently use generalized Rayleigh quotients $\rho(x) := \frac{x^*H_1x}{x^*H_2x}$, with $x\in\C^n\setminus\{0\}$, for Hermitian positive semidefinite $H_1$ and $H_2\in\C^{n,n}$,
where we define that
\begin{equation}\label{eq:defrq}
	\rho(x)= \frac{x^*H_1x}{x^*H_2x} :=0, \quad\text{if both $x^*H_1x = 0$ and $x^*H_2x= 0$.}
\end{equation}
By \eqref{eq:defrq}, the function $\rho(x)$ is well-defined and lower semi-continuous, i.e., $\liminf_{y\to x}\rho(y)\geq \rho(x)$, for all $x\in\C^n\setminus\{0\}$, so that we can properly define optimization problems involving such Rayleigh quotients.

\subsection{Spectral properties of general DAEs}\label{sec:gendae}

For DAEs $E\dot x=Ax$  we synonymously speak about the pair of matrices $(E,A)$ and the pencil $\lambda E-A$. Then the structural properties are characterized via the Kronecker canonical form \cite{Gan59a}.
\begin{theorem}\label{th:kcf}
Let $E,A\in {\mathbb C}^{n,m}$. Then there exist nonsingular matrices
$S\in {\mathbb C}^{n,n}$ and $T\in {\mathbb C}^{m,m}$ such that
\begin{equation}\label{kcf}
S(\lambda E-A)T=\diag({\mathcal L}_{\epsilon_1},\ldots,{\mathcal L}_{\epsilon_p},
{\mathcal L}^\top_{\eta_1},\ldots,{\mathcal L}^\top_{\eta_q},
{\mathcal J}_{\rho_1}^{\lambda_1},\ldots,{\mathcal J}_{\rho_r}^{\lambda_r},{\mathcal N}_{\sigma_1},\ldots,
{\mathcal N}_{\sigma_s}),
\end{equation}
where the block entries have the following properties:
\begin{enumerate}
\item[\rm (i)]
Every entry ${\mathcal L}_{\epsilon_j}$ is a bidiagonal block of size
${\epsilon_j}\times ({\epsilon_j+1})$, $\epsilon_j\in{\mathbb N}_0$,
of the form
\[
\lambda\left[\begin{array}{cccc}
1&0\\&\ddots&\ddots\\&&1&0
\end{array}\right]-\left[\begin{array}{cccc}
0&1\\&\ddots&\ddots\\&&0&1
\end{array}\right].
\]
\item[\rm (ii)]
Every entry ${\mathcal L}^\top_{\eta_j}$ is a bidiagonal block of size
$({\eta_j+1})\times {\eta_j}$, $\eta_j\in{\mathbb N}_0$,
of the form
\[
\lambda\left[\begin{array}{ccc}
1\\0&\ddots\\&\ddots&1\\&&0
\end{array}\right]-
\left[\begin{array}{ccc}
0\\1&\ddots\\&\ddots&0\\&&1
\end{array}\right].
\]
\item[\rm (iii)]
Every entry ${\mathcal J}_{\rho_j}^{\lambda_j}$ is a Jordan block of size
${\rho_j}\times{\rho_j}$, $\rho_j\in{\mathbb N}$, $\lambda_j\in{\mathbb C}$,
of the form
\[
\lambda\left[\begin{array}{cccc}
1\\&\ddots\\&&\ddots\\&&&1
\end{array}\right]-
\left[\begin{array}{cccc}
\lambda_j&1\\&\ddots&\ddots\\&&\ddots&1\\&&&\lambda_j
\end{array}\right].
\]
\item[\rm (iv)]
Every entry ${\mathcal N}_{\sigma_j}$ is a nilpotent block of size
${\sigma_j}\times {\sigma_j}$, $\sigma_j\in{\mathbb N}$,
of the form
\[
\lambda\left[\begin{array}{cccc}
0&1\\&\ddots&\ddots\\&&\ddots&1\\&&&0
\end{array}\right]-
\left[\begin{array}{cccc}
1\\&\ddots\\&&\ddots\\&&&1
\end{array}\right].
\]
\end{enumerate}
The Kronecker canonical form is unique up to permutation of the blocks.
\end{theorem}
The sizes of the rectangular blocks, $\eta_j$ and $\epsilon_i$,
are called \emph{left and right minimal indices} of $\lambda E-A$, respectively. The matrix pencil $\lambda E-A$ with $E,A\in\mathbb C^{n,m}$ is called \emph{regular} if $n=m$ and
$\operatorname{det}(\lambda_0 E-A)\neq 0$ for some $\lambda_0 \in \mathbb C$,
otherwise it is called \emph{singular}.

A value $\lambda_0\in\mathbb C$ is called a (finite) eigenvalue of matrix pencil $\lambda E-A$ or the matrix pair $(E,A)$ if
\[
\operatorname{rank}(\lambda_0E-A)<\max_{\alpha\in\mathbb C}
\operatorname{rank}(\alpha E-A).
\]
Furthermore,
$\lambda_0=\infty$ is said to be an eigenvalue of $\lambda E-A$ if zero is an eigenvalue of $\lambda A-E$.
The size of the largest block ${\mathcal N}_{\sigma_j}$ is
called the \emph{index} $\nu$ of the pencil $\lambda E-A$, where, by convention,  $\nu=0$ if $E$ is invertible.


For matrix pencils associated with
dHDAE systems we recall the  special Kronecker structure in the next subsection.
\subsection{Spectral Properties of dHDAEs}\label{sec:dhdaes}
The spectral properties of the \emph{dH pencil} $P(\lambda)=\lambda E-(J-R)Q$ associated with~\eqref{dhdae} have been completely characterized in \cite{MehMW18}, see also \cite{GilMS18,MehMS16} for partial results.
\begin{theorem}\label{thm:singindspec}
Let $E,Q\in \mathbb C^{n,m}$ satisfy $E^* Q=Q^* E\geq 0$ and let all left minimal indices of
$\lambda E-Q$ be equal to zero (if there are any). Furthermore, let $R\geq 0$. Then the following
statements hold for the pencil $P(\lambda)=\lambda E-(J-R)Q$.
\begin{enumerate}
\item[\rm (i)] If $\lambda_0\in\mathbb C$ is a finite eigenvalue of $P(\lambda)$ then $\operatorname{Re}(\lambda_0)\leq 0$.
\item[\rm (ii)] If $\omega\in\mathbb R\setminus\{0\}$ and $\lambda_0=i\omega$ is a finite eigenvalue of $P(\lambda)$, then
$\lambda_0$ is semisimple. Moreover, if the columns of $V\in\mathbb C^{m,k}$ form a basis of a regular deflating
subspace of $P(\lambda)$ associated with $\lambda_0$, then $RQV=0$.
\item[\rm (iii)] The index of $P(\lambda)$ is at most two.
\item[\rm (iv)] All right minimal indices of $P(\lambda)$ are at most one (if there are any).
\item[\rm (v)] If in addition $\lambda E-Q$ is regular, then all left minimal indices of $P(\lambda)$ are zero
(if there are any).
\end{enumerate}
\end{theorem}

It has been demonstrated in \cite{MehMW18} that the   condition $E^* Q=Q^* E\geq 0$ cannot be weakend in general. Furthermore, if  the pencil $\lambda E-Q$ is singular,  then the dH pencil $\lambda E-(J-R)Q$ may have eigenvalues with positive real part.

\subsection{Removing the \texorpdfstring{$Q$}{Q} factor in linear pHDAE systems}
\label{sec:noQ}
In many applications, where dHDAE systems arise, one has that $Q = I$. In this case $E=E^*\geq 0 $ and  the Hamiltonian is given by $\mathcal H=\frac 12 x^*Ex$. This representation has many advantages, since the coefficients appear linearly in \eqref{dhdae}, which greatly simplifies the analysis and also the perturbation theory. 

If $Q$ is invertible, then the state equation can be  multiplied with $Q^*$ from the left without changing the solution set, yielding a system with the same solution set given by
\begin{equation}\label{mwithQ}
	\widetilde{E}\dot x=Q^* E\dot x = Q^* (J-R)Q x=
 (\widetilde{J}-\widetilde{R})x,   \end{equation}
with
$\widetilde{E} \vcentcolon= Q^* E$,
$\widetilde{J} \vcentcolon=Q^* JQ$, $\widetilde{R} \vcentcolon= Q^* R Q$, which
is again a dHDAE, but now has $\widetilde{Q}=I$.

If $Q$ is not of full rank, then the situation is more complex and even in the case $E=I$ the solution can grow unboundedly as the following example from \cite{MehMW18} shows. Consider
\begin{displaymath}
    \begin{bmatrix} \dot{x}_1 \\ \dot{x}_2 \end{bmatrix} = JQ\begin{bmatrix}x_1\\x_2\end{bmatrix} = \begin{bmatrix} 0&-1\\ 1&0\end{bmatrix}\begin{bmatrix} 1&0\\ 0&0\end{bmatrix} \begin{bmatrix}x_1\\x_2\end{bmatrix},\quad
    \begin{bmatrix}x_1(0)\\x_2(0)\end{bmatrix} = \begin{bmatrix}x_{1,0}\\x_{2,0}\end{bmatrix}
\end{displaymath}
with Hamiltonian $\mathcal H=\tfrac {1}{2} x_1^2$. It has the solution $x_1=x_{1,0}$, $x_2=x_{2,0} +t x_{1,0}$ and thus has linear growth and hence is not stable, since the system matrix $JQ$ is a Jordan block of size $2$ at the eigenvalue $0$.

In view of this discussion, to have robust asymptotic stability, we assume in the following that $Q$ is invertible and in this case we may as well just assume that the dHDAE system has the form
\begin{equation}\label{dhdae1}
E\dot x=(J-R) x,
\end{equation}
with $E=E^*\geq 0$, $R=R^*\geq 0$, $J=-J^*$. Actually one should use this representation  even in the case of ordinary differential equations, see \cite{MehU23} for a detailed discussion.

\subsection{A staircase form for dHDAEs}\label{sec:staircase}

To check whether a system of the form \eqref{dhdae1} is robustly asymptotically stable
one can use the following staircase form under unitary transformations from \cite{AchAM21}.
\begin{lemma}\label{lem:tSF}
Let $E,J,R\in\mathbb C^{n,n}$ satisfy $E=E^*\geq 0$, $R=R^*\geq 0$ and $J=-J^*$.
Then there exists a unitary matrix $P\in\mathbb C^{n,n}$, such that
\begin{eqnarray}
\hat E :=P^*\ E\ P
&=:&\begin{bmatrix}
E_{1,1} & E_{2,1}^* & 0 & 0 & 0 \\
E_{2,1} & E_{2,2} & 0 & 0 & 0 \\
0  & 0 & 0 & 0 & 0 \\
0  & 0 & 0 & 0 & 0 \\
0  & 0 & 0 & 0 & 0
\end{bmatrix}, \nonumber\\
\hat R :=P^*\ R\ P
&=:&\begin{bmatrix}
R_{1,1} & R_{2,1}^* & R_{3,1}^*& 0 & 0 \\
R_{2,1} & R_{2,2} & R_{3,2}^* &0 & 0  \\
R_{3,1} & R_{3,2} & R_{3,3} &0 & 0 \\
0  & 0 & 0 & 0 & 0 \\
0  & 0 & 0 & 0 & 0
\end{bmatrix}, \label{staircase:EJR} \\
\hat J :=P^*\ J\ P
&=:&\begin{bmatrix}
J_{1,1} & -J_{2,1}^* & -J_{3,1}^* & -J_{4,1}^* & 0\\
J_{2,1} & J_{2,2} & -J_{3,2}^*& 0 & 0\\
J_{3,1} & J_{3,2} & J_{3,3}&  0 & 0\\
J_{4,1} & 0& 0 & 0 &0\\
 0 & 0 & 0 & 0 & 0
\end{bmatrix}.
\nonumber
\end{eqnarray}
The matrices are partitioned in the same way, with (square) diagonal block matrices of sizes $n_1$, $n_2$, $n_3$, $n_4=n_1$, $n_5$.
If present, then $
\begin{bmatrix} E_{1,1} & E_{2,1}^* \\ E_{2,1} & E_{2,2}  \end{bmatrix}>0
$, and the matrices $J_{4,1}$, $J_{3,3}-R_{3,3}$ with $R_{3,3}\geq 0$ are invertible.
\end{lemma}

\begin{corollary}\label{cor} Let $E,J,R\in\mathbb C^{n,n}$ with $E=E^*\geq 0$, $R=R^*\geq 0$ and $J=-J^*$  be in staircase form \eqref{staircase:EJR}.

\begin{itemize}
\item [a)] The pencil $\lambda E-(J-R)$ is regular if and only if $n_5=0$ if and only if  $\ker \mat{c} E\\ J\\ R\rix\neq \{0\}$.
\item [b)] The pencil is regular and of index at most one if and only if $n_1=n_4=n_5=0$.
\item [c)] The finite eigenvalues of $\lambda E-(J-R)$ are the eigenvalues of\\ $\lambda E_{2,2}- (-J_{3,2}^*-R_{3,2}^*)(J_{3,3}-R_{3,3})^{-1} (J_{3,2}-R_{3,2})$.
\end{itemize}
\end{corollary}
In \cite{AchAM21} also the following refinement of the staircase form \eqref{staircase:EJR} has been presented.
\begin{corollary}\label{cor:almostkcf}
Let $E,J,R\in\mathbb C^{n,n}$ satisfy $E=E^*\geq 0$, $R=R^*\geq 0$ and $J=-J^*$.
Then there exist nonsingular matrices  $L,Z\in\mathbb C^{n,n}$, such that
\begin{eqnarray}
\tilde E :=L\ E\ Z
&=:&\begin{bmatrix}
E_{1,1} & 0& 0 & 0 & 0 \\
0&  E_{2,2}& 0 & 0 & 0 \\
0  & 0 & 0 & 0 & 0 \\
0  & 0 & 0 & 0 & 0 \\
0  & 0 & 0 & 0 & 0
\end{bmatrix}, \nonumber\\
\tilde R :=L\ R\ Z
&=:&\begin{bmatrix}
0 & 0 & 0& 0 & 0 \\
0 & R_{2,2}& 0 &0 & 0  \\
0 & 0 & I_{n_3} &0 & 0 \\
0  & 0 & 0 & 0 & 0 \\
0  & 0 & 0 & 0 & 0
\end{bmatrix}, \label{kronecker:EJR} \\
\tilde J :=L\ J\ Z
&=:&\begin{bmatrix}
0 & 0 & 0 & -I_{n_4} & 0\\
0 & J_{2,2} & 0& 0 & 0\\
0& 0 & 0&  0 & 0\\
I_{n_4} & 0& 0 & 0 &0\\
 0 & 0 & 0 & 0 & 0
\end{bmatrix},
\nonumber
\end{eqnarray}
with $E_{11}>0$ diagonal, $E_{22}>0$, and $R_{22}\geq 0$.
\end{corollary}

Based on the spectral properties of dHDAE systems, in the next subsection we give a characterization of robust asymptotic stability.
\subsection{Robust asymptotic stability}\label{sec:robasystab}
The following result gives a characterization of robust asymptotic stability for dHDAE systems.
\begin{theorem}\label{thm:stabledh}
    Consider the  dH
    pencil $\lambda E-(J-R)$ associated with the dHDAE~\eqref{dhdae1}. Let $N(E)$ be a matrix with columns that span the right nullspace of $E$. Then the system is robustly asymptotically stable if and only if the following conditions hold:
    \begin{itemize}
        \item [a)] the pencil is regular;
        \item [b)] the index of the pencil is at most one and every principal submatrix of  $N(E)^*(J-R)N(E)$ is nonsingular;
        \item [c)] the maximal real part $\mu$ of a finite eigenvalue of $\lambda E-(J-R)$ (the spectral abscissa) is negative.
    \end{itemize}
    The boundary of the set of robustly asymptotically stable pencils are those pencils where one of the three conditions is violated.
\end{theorem}
\proof
If the three conditions a)-c) hold, then in \eqref{staircase:EJR} we have that $n_1=n_4=n_5=0$, $E_{22}=E_{22}^*>0$ and $J_{33}-R_{33}$ is nonsingular. Sufficiently small structured perturbations will not destroy the nonsingularity of these two matrices. So it remains to study the case that the rank of $E$ is increased. Since all principal submatrices of $N(E)^*(J-R)N(E)$ are nonsingular, an increase of the rank of $E$ will again lead to a system of index at most one.

For the converse direction we show that a violation of any of the three conditions yields a pair on  the boundary of the robustly asymptotically stable dH pencils.

a) If the pencil is singular, then it has
a zero row and column in the staircase form. Any arbitrary small positive perturbation in the corresponding
diagonal position of $E$ leads to an eigenvalue $0$ of the pencil, hence the pair is on the boundary of the robustly asymptotically stable dH pencils.

b) If the index is bigger than  one, then in \eqref{kronecker:EJR} we have $n_1=n_4>0$. Looking at the structure of the subpencil consisting  of the first and forth block row and column, we have $n_4$ copies of
Example~\ref{ex:perind} and since $n_4>0$ a similar perturbation  will put the pencil on the boundary of the asymptotically stable systems.

c) If $\mu=0$ then the system already has purely imaginary eigenvalues.
\eproof

\begin{remark} \label{rem:R22posdef}{\rm
A sufficient condition for the second part of b) in Theorem~\ref{thm:stabledh} is that $N(E)^*RN(E)>0$ because then every principal submatrix has this property as well.
}
\end{remark}
By structured perturbations of dHDAEs, finite eigenvalues in the open left half plane may move to the imaginary axis, and the regularity or the index of the system may change. 
In the next subsection we present distance measures under structured perturbations that characterize when this happens.
\subsection{Structured distances}\label{sec:structdist}\

In this subsection, we consider perturbations in the coefficient matrices $E$, $J$ and $R$ of a dHDAE system of the form~\eqref{dhdae1}. These take the form
\begin{equation}
    (\tilde E, \tilde J-\tilde R)=( E+\Delta_E, J+\Delta_J - (R+\Delta_R)).
\end{equation}
In order to measure these perturbations in $E$, $J$ and $R$, we consider the following norm for a tuple $(\Delta_E,\Delta_J,\Delta_R)\in (\C^{n,n})^3$ given by
\begin{equation}
    \nnrm{(\Delta_E,\Delta_J,\Delta_R)}=\sqrt{\|\Delta_E\|^2+\|\Delta_J\|^2+\|\Delta_R\|^2}.
\end{equation}
%
We study the following structured distances.

\begin{definition}\label{def:distances}
	Consider a robustly  asymptotically stable dHDAE system of the form~\eqref{dhdae1}, and the following sets  of structured perturbations to the matrices $E, J$ and $R$,
	\begin{align}
		\mathcal S_d(E,J,R)&:= \{(\Delta_E,\Delta_J,\Delta_R):\Delta_E\leq 0, E+\Delta_E\geq 0, \Delta_J^*=-\Delta_J, \Delta_R\leq 0, R+\Delta_R\geq 0 \}, \label{set:Sd}\\
		\mathcal S_i(E,J,R)&:= \{(\Delta_E,\Delta_J,\Delta_R): \Delta_E^*=\Delta_E,E+\Delta_E\geq 0, \Delta_J^*=-\Delta_J, \Delta_R^*=\Delta_R, R+\Delta_R\geq 0 \}.\label{set:Si}
	\end{align}
	Let $\mathbb S\in \{\mathcal S_d,\mathcal S_i\}$. Then, we consider the  distances:
	\begin{enumerate}
		\item The distance to singularity $d_{sing}^{\mathbb S}(E,J,R)$ under perturbations from  $\mathbb S\in \{\mathcal S_d,\mathcal S_i\}$, is defined by
		\begin{equation}\label{def:dist-sing}
			d_{sing}^{\mathbb S}(E,J,R) := \inf\{\nnrm{(\Delta_E,\Delta_J,\Delta_R)}~:~(E+\Delta_E,J+\Delta_J-R-\Delta_R) \text{ is singular}\}.
		\end{equation}
		\item The distance to higher index $d_{hi}^{\mathbb S}(E,J,R)$ under perturbations from  $\mathbb S\in \{\mathcal S_d,\mathcal S_i\}$, is defined by
		\begin{equation}\label{def:dist-hi}
			d_{hi}^{\mathbb S}(E,J,R) := \inf\{\nnrm{(\Delta_E,\Delta_J,\Delta_R)}~:~(E+\Delta_E,J+\Delta_J-R-\Delta_R) \text{ has index }>1\}.
		\end{equation}
		\item The distance to a system with eigenvalues on the imaginary axis $d_{im}^{\mathbb S}(E,J,R)$ under perturbations from the set $\mathbb S\in \{\mathcal S_d,\mathcal S_i\}$, is defined by
		\begin{equation}\label{def:dist-im}
			d_{im}^{\mathbb S}(E,J,R) := \inf\{\nnrm{(\Delta_E,\Delta_J,\Delta_R)}~:~\Lambda(E+\Delta_E,J+\Delta_J-R-\Delta_R)\cap i\mathbb R \neq \emptyset \}.
		\end{equation}
	\end{enumerate}
\end{definition}
The  additional constraint $\Delta_R\leq 0$ in $S_d(E,J,R)$ is motivated by the observation that for the case $E=I$, under the requirement $R+\Delta_R\geq 0$, the optimal perturbation attaining the distance to instability is generally of rank two. In this case imposing $\Delta_R\leq 0$ allows for a rank-one characterization of the distance to instability while still preserving the dissipative structure $R+\Delta_R\geq 0$, see \cite{BenMPS25} for a detailed discussion.

By  Theorem~\ref{thm:stabledh}, the structured distance  of a robustly asymptotically stable dHDAE system to the boundary of this set is then given by
\begin{equation*}\label{def:dist-inst}
	d_{inst}^{\mathbb S}(E,J,R):= \min \{d_{sing}^{\mathbb S}(E,J,R), d_{hi}^{\mathbb S}(E,J,R),d_{im}^{\mathbb S}(E,J,R)\}.
\end{equation*}
%

In addition to full perturbations of the dHDAE system, we will also consider partial perturbations, e.g.
the partial perturbation sets $\mathcal S_d(J,R):=\mathcal S_d(0,J,R)$ and $\mathcal S_i(J,R):= \mathcal S_i(0,J,R)$.
%
%

\subsection{Checking robust asymptotic stability}\label{sec:distalg}

Based on  Theorem~\ref{thm:stabledh},
we have the following procedure to check, via a numerically robust algorithm, whether for a dHDAE of the form \eqref{dhdae1} one of the conditions  for robust asymptotical stability is violated. This procedure also leads  to bounds for the distance to the nearest dHDAE which is on the boundary of the set of robustly asymptotically stable dHDAEs.

{\bf \label{alg:dist} Algorithm to check robust asymptotic stability}

Step 1: Check whether $E$ is invertible by computing the 
spectral decomposition $E=Q_e \Lambda_E Q_e^*$ and if $E$ is singular then determine the right nullspace $N(E)$.\\
The distance to the nearest dHDAE system with infinite eigenvalues then is   $d_{DAE}^{\mathbb S}:=\lambda_{\min}(E)$.

Step 2: Check the regularity of the pencil by computing the singular value decomposition
\[
U_s^*\mat{c} E\\ J\\ R\rix V_s=\Sigma_s.
\]
The distance to the nearest singular dissipative Hamiltonian system then is
\[
d_{sing}^{\mathbb S}:=\sigma_{\min}\left (\mat{c} E\\ J\\ R\rix \right ).
\]
This distance can also be determined via the algorithm in \cite{GugM22}.

Step 3: Check the index by computing the singular value decomposition
\[
U^*_iN(E)^*(J-R) N(E)V_i=\Sigma_i.
\]
The distance to the nearest system of index two is then bounded via $d_{hi}^{\mathbb S}\geq \sigma_{\min}(N(E)^* (J-R) N(E))$, see Theorem~\ref{thm:dist-hiEJR} for an exact characterization.

Step 4. Compute the spectral decomposition
\[
U_r^*N(E)^*RN(E)U_r=\Lambda_r.
\]
If $\Lambda_r$ is nonsingular, then  $\Delta_R=-\lambda_{\min} (\Lambda_r) U_rx_r x_r^*U_r^*$ is the smallest perturbation such that $R+\Delta_R$ is singular. Here $x_r$ is a normalized eigenvector to the minimal eigenvalue $\lambda_{\min}(\Lambda_r)$.\\
The distance to the nearest system with singular $R$ coefficient is then $d_r^{\mathbb S}:=\lambda_{\min}(\Lambda_r)$.\\

Step 5. Compute the Schur complement pencil associated with the finite eigenvalues
from the staircase form  in Lemma~\ref{lem:tSF} or the orthogonal condensed form in \cite{MehX15}. Let
\[
P_f(\lambda):=\lambda E_f-(J_f-R_f)=\lambda E_{2,2}- (-J_{3,2}^*-R_{3,2}^*)(J_{3,3}-R_{3,3})^{-1} (J_{3,2}-R_{3,2}).
\]

Determine the smallest perturbation from the corresponding sets
$\mathcal S_d(E_f,J_f,R_f)$, $\mathcal S_i(E_f,J_f,R_f)$
associated with the reduced pencil $\lambda E_f-(J_f-R_f)$ for which an eigenvalue of $P_f$ reaches the imaginary axis
\[
d_{im}^{\mathbb S}(E_f,J_f,R_f).
\]
This procedure directly determines an upper bound for the boundary of the set of robustly asymptotically stable dHDAE systems, since we can definitely destroy the robust asymptotic stability by a perturbation of the size $
\min(d_{im}^{\mathbb S},d_{sing}^{\mathbb S},d_{im}^{\mathbb S})$.
The different bounds are related via  the inequalities
\begin{eqnarray*}
&& \quad \max (\sigma_{\min} (E),\sigma_{\min} (N(E)^*(J-R)N(E)),d_r^{\mathbb S})\leq d_{hi},
\\
&&\max (\sigma_{\min} (E),\sigma_{\min} (J),\sigma_{\min} (R))\leq d_{sing}^{\mathbb S} = \sigma_{\min} \left (\mat{c}  E\\ J\\ R\rix \right ).
\end{eqnarray*}
We also have the following sufficient
conditions.
\begin{theorem}\label{thm:Rposdef}
  Consider the pencil $\lambda E-(J-R)$ associated with the dHDAE~\eqref{dhdae1} subject to a structured perturbation $(\Delta_E, \Delta_J, \Delta_R)$.
If $R>0$ and $R+\Delta_R>0$, then the perturbed dHDAE system is robustly asymptotically stable.
\end{theorem}
\proof
Since $J+\Delta_J-R-\Delta_R$ is invertible if $R+\Delta_R$ is positive definite, the pencil is robustly regular and of index at most one, since also every principal submatrix of $R+\Delta_R$ is positive definite, see Remark~\ref{rem:R22posdef}.
Let $\hat N$ be a matrix with columns that span the orthogonal complement of $N(E+\Delta_E)$.
Then the subpencil associated with the finite eigenvalues, given by
\[
\lambda \tilde E-(\tilde J-\tilde R):=\lambda {\hat N}^*(E+\Delta_E) {\hat N}-{\hat N}^*(J+\Delta_J-R-\Delta_R) {\hat N},
\]
has only eigenvalues with negative real part, on the left of the line
\[-\lambda_{\min}(\tilde E^{-1/2})^*\tilde R(\tilde E^{-1/2})i.\]
Hence,  the perturbed pencil is robustly asymptotically stable.
\eproof

Thus, to make a dHDAE system to be not robustly asymptotically stable we need to choose $\Delta_R$ at least such that $R+\Delta_R$ is singular.

Another easy bound in the case of real dHDAE systems is presented in \cite{MehMW21}.
%
%
%
%
%
%
%
%
    %
%
%
We will extend this result to the complex case below.

\subsection{Mapping results}\label{sec:mapping}
In this subsection, we recall some mapping results from the literature in a form that allows a direct application in computing the structured distances in this paper.
The following result from~\cite{BorKMS14,MacMT08} gives minimal norm solutions to the Hermitian mapping problem with respect to the spectral norm.

%
\begin{theorem}\label{map:herm}
	Let $x \in \C^n\setminus\{0\}$ and $y\in \C^{n}$. Then
	\begin{itemize}
		\item [a)] there exists a Hermitian matrix $H \in \C^{n,n}$ such that $Hx = y$ if and only if $\Im{(x^*y)} = 0$ and we have
		\[
		\min\big\{\|H\|\; : H\in \C^{n,n},\; H^*=H, \; Hx=y\big\} = \frac{\|y\|}{\|x\|},
		\]
		and the minimum is attained by
		\begin{equation}\label{def:hatH}
			\hat H_{(x,y)} := \frac{\|y\|}{\|x\|}\mat{cc}\frac{y}{\|y\|}&\frac{x}{\|x\|} \rix \mat{cc}\frac{y^*x}{\|x\|\|y\|}&1\\1&\frac{x^*y}{\|x\|\|y\|} \rix^{-1}\mat{cc}\frac{y}{\|y\|}&\frac{x}{\|x\|} \rix^*
		\end{equation}
		if $x$ and $y$ are linearly independent and by $\hat H_{(x,y)}:=\frac{yx^*}{x^*x}$, otherwise;
		\item [b)] there exists a skew-Hermitian matrix $S \in \C^{n,n}$ such that $Sx = y$ if and only if $\Re{(x^*y)} = 0$ and we have
		\[
		\min\big\{\|S\|\; :\; S ^*=-S, \; Sx=y\big\} = \frac{\|y\|}{\|x\|},
		\]
		and the minimum is attained by $\hat S:=-i\hat H_{(x,iy)}$, where $\hat{H}$ is defined in~\eqref{def:hatH}.
	\end{itemize}
\end{theorem}

\begin{corollary}\label{map:herm-corr}
Let $x,y\in \C^{m}\setminus \{0\}$ and $B\in\mathbb{C}^{n,m}$ satisfy $B^*B=I_m$. Then
\begin{enumerate}
    \item[a)] there exists a Hermitian matrix $\Delta$ such that $B^*\Delta B x = y$ if and only if $\Im{(x^*y)}=0$ and we have
    \[
    \min\{\|\Delta\|~:~\Delta\in \C^{n,n},~ \Delta^*=\Delta,~ B^*\Delta B x = y\} = \frac{\|y\|}{\|x\|},
    \]
    and the minimum is attained by $\hat \Delta = B\hat H_{(x,y)}B^*$, where $\hat H$ is defined in~\eqref{def:hatH},
    \item[b)] there exists a skew-Hermitian matrix $\Delta$ such that $B^*\Delta B x = y$ if and only if $\Re{(x^*y)}=0$ and we have
    \[
    \min\{\|\Delta\|~:~\Delta\in \C^{n,n},~ \Delta^*=-\Delta,~ B^*\Delta B x = y\} =\frac{\|y\|}{\|x\|}.
    \]
    and the minimum is attained by $\hat \Delta =-iB\hat H_{(x,iy)}B^*$, where $\hat H$ is defined in~\eqref{def:hatH}.
\end{enumerate}
\end{corollary}
\begin{proof}
We only prove part a), the proof of part b) follows similarly.
If there exists $\Delta \in\C^{n,n}$ such that $\Delta^*=\Delta$ and $B^*\Delta B x = y$, then $x^*y= x^*B^*\Delta B x \in \R$. Conversely, if $x^*y\in \R$ then $\Delta = B \hat H_{x,y} B^*$, where $\hat H$ is defined in~\eqref{def:hatH} satisfies $\Delta^*=\Delta$ and $B^*\Delta B x = y$.

Note that, $\|y\|=\|B^*\Delta B x \| \leq \|B\|^2\|\Delta\|\|x\| = \|\Delta\|\|x\|.$ This gives,
\begin{align}
    \frac{\|y\|}{\|x\|}~\leq~ \min\{\|\Delta\|~:~\Delta\in \C^{n,n},~\Delta^*=\Delta,~B^*\Delta B x = y \}.\label{eq:map1}
\end{align}
For the choice $\Delta=B{\hat H_{(x,y)}}B^*,$ where $\hat H$ is defined in~\eqref{def:hatH} we have $\Delta^*=\Delta,~ B^*\Delta B x = y$, and $\|\Delta\|=\|{B\hat H_{x,y}B^*}\|=\|y\|/\|x\|$, showing equality in~\eqref{eq:map1}.
\end{proof}

We close this section with the following result~\cite[Theorem 2.3]{MehMS16} that gives minimal norm solutions to the Hermitian positive semidefinite mapping problem with respect to the spectral norm.
\begin{theorem}\label{map:def}
    Let $x,y \in \C^n \setminus \{0\}$. Then there exists a positive semidefinite Hermitian matrix $H$ such that $Hx=y$ if and only if $x^*y\in \R$ and $x^*y > 0$. If the latter condition is satisfied, then
    \[
    \min\left \{\|H\|:H \in \C^{n,n},H^*=H \geq 0,~Hx=y\right \}=\frac{{\|y\|}^2}{x^*y}
    \]
    and the minimum is attained for the rank one matrix
    $\tilde H=\frac{1}{x^*y}yy^*$.
\end{theorem}

\begin{corollary}\label{map:def-corr}
    Let $x,y \in \C^m \setminus \{0\}$ and $B\in \C^{n,m}$ satisfy $B^*B=I_m$. Then there exists a positive semidefinite Hermitian matrix $\Delta$ such that $B^*\Delta B x = y$ if and only if $x^*y\in \R$ and $x^*y>0$. If the later condition is satisfied, then
    \[
    \min\left \{\|\Delta\|:\Delta\in \C^{n,n},\Delta^*=\Delta \geq 0,~B^*\Delta Bx=y\right \}=\frac{{\|y\|}^2}{x^*y}
    \]
    and the minimum is attained for the rank one matrix
    $\tilde \Delta=\frac{1}{x^*y}Byy^*B^*$.
\end{corollary}
\begin{proof}
    Let $\Delta$ be a Hermitian positive semidefinite matrix satisfying $B^*\Delta Bx=y$, then
    \[
    \|y\|^2 = \|B^*\Delta B x\|^2 = x^*(B^*\Delta B)^*(B^*\Delta B)x \leq \|B^*\Delta B\| x^*B^*\Delta Bx\leq \|\Delta\|x^*y,
    \]
    where we used the fact that $C^2\leq \|C\|C$ for a Hermitian positive semidefinite matrix $C$. Hence we have
    \[
    \|\Delta \| \geq \frac{\|y\|^2}{x^*y},
    \]
    and the bound is attained by the matrix $\hat \Delta = \frac{1}{x^*y}Byy^*B^*$.
\end{proof}

\noindent In this section we have summarized some preliminary results as well as the characterization of robust asymptotical stability. In the following section we  characterize the different distances. 

\section{Distance to the nearest dHDAE that is not robustly asymptotically stable under full perturbations}\label{sec:dist}

In this section we discuss in detail the various structured distances in Definition~\ref{def:distances}. We derive explicit characterizations as well as bounds.

\subsection{Distance to nearest dHDAE with purely imaginary eigenvalues}\label{sec:distim}
In this section, we present analytic results that characterize the smallest structured distance to a dHDAE system $(E,J-R)$ of the form \eqref{dhdae1} with purely imaginary eigenvalues.

We use the following notation. Let
\begin{equation}\label{set:Lambda}
    \Lambda:=\{\omega:\det(i\omega E-J) =0\}
\end{equation}
and for $\omega\in \R\setminus \Lambda$, define $M=M(\omega):=(i\omega E-J)^{-1}$ and the  block matrices
\begin{equation}\label{mat:tildeH}
    \tilde H_1:=\mat{cc} 0 & -M^*\\ M & 0 \rix, \quad \tilde H_2:=\mat{cc} 0 & -i\omega M \\ i\omega M^* & 0 \rix,
\end{equation}
as well as
\begin{equation}\label{mat:G1G2}
    G_1:=\mat{c} I\\i\omega I \rix M^*R^2M \mat{cc} I & -i\omega I \rix + \mat{cc}I & \\ & I \rix, \quad G_2:=\mat{c} I\\i\omega I \rix M^*M \mat{cc}I & -i\omega I \rix.
\end{equation}
Note that $G_1$ is positive definite, and hence has a unique positive definite square root $G_1^{1/2}$. Thus, we can define the following matrices
\begin{equation}\label{mat:GH1H2}
    G:=G_1^{-1/2}G_2G_1^{-1/2}, \quad H_1:=G_1^{-1/2}\tilde H_1G_1^{-1/2}, \text{ and } H_2:= G_1^{-1/2}\tilde H_2G_1^{-1/2}.
\end{equation}
For our analysis, we will make use of the following Lemma.
\begin{lemma}\label{lem:tfae-map}
    Consider a robustly asymptotically stable dHDAE system of the form~\eqref{dhdae1}, and let $\mathbb S\in \{\mathcal S_d,\mathcal S_i\}$, where $\mathcal S_d,\mathcal S_i$ are defined in~\eqref{set:Sd} and~\eqref{set:Si}, respectively. Furthermore, let $\omega\in \R$ be such that
    $M:=(i\omega E-J)^{-1}$ exists. Then for $(\Delta_E,\Delta_J,\Delta_R)\in\mathbb S$, the following are equivalent:
    \begin{enumerate}
     \item[\rm (1)] ${\rm det}\left(i\omega (E+\Delta_E)-(J+\Delta_J)+(R+\Delta_R)\right)=0$.
        \item[\rm (2)] $\exists~x(\neq 0)\in \C^{n}$ s.t. $(R+\Delta_R)x=0$, $(i\omega (E+\Delta_E)-(J+\Delta_J))x=0.$
        \item[\rm (3)] $\exists~v_E,v_J\in \C^{n}$~ s.t. $v:=-i\omega v_E+v_J\neq 0$, ~$\Delta_R M v = -RM v, \Delta_EM v=v_E$, and $\Delta_JM v=v_J$.
    \end{enumerate}
\end{lemma}
\proof
The equivalence $(1)\iff (2)$ follows immediately from Theorem~\ref{thm:singindspec}.

$(2)\implies(3)$. If $(2)$ holds, then with $v_J=\Delta_J x$ and $v_E=\Delta_Ex$ we have
\begin{equation*}
 (i\omega E-J )x=-i\omega \Delta_E x+\Delta_J x=-i\omega v_E+v_J=:v
\end{equation*}
Since $M=(i\omega E-J)^{-1}$ is nonsingular, we have $v\neq0$ and $x=M v$. This implies that $\Delta_R M v = -RM v$, $\Delta_EM v=v_E$, and $\Delta_JM v=v_J$.

$(3)\implies(2)$ Suppose that $(3)$ holds and set $x=Mv$. Then clearly $(R+\Delta_R)x=RMv+\Delta_R Mv=RMv-RMv=0$ and
\[
(i\omega (E+\Delta_E)-(J+\Delta_J))x= (i\omega E-J)Mv + (i\omega \Delta_EMv-\Delta_JMv)=v-v=0.
\]
This implies (2).
\eproof

Using Lemma~\ref{lem:tfae-map} we have the following bounds for the distance to the nearest dHDAE system with purely imaginary eigenvalues.
\begin{theorem}\label{thm:dist-imagEJR}
    Consider a robustly asymptotically stable dHDAE system of the form~\eqref{dhdae1}. Then,
    \begin{enumerate}
        \item  for  perturbations from $\mathcal S_d$, the distance to a dH pencil with purely imaginary eigenvalues
        is bounded by
        \begin{align}\label{Sd:dist-imagEJR}
            & (d_{im}^{\mathcal S_d}(E,J,R))^2\geq \min\Bigg\{ \inf_{x\in \Omega}\frac{\|Rx\|^4}{(x^*Rx)^2},\nonumber \\ & \inf_{\omega\in \R\setminus \Lambda}  \left\{\inf_{\substack{u\in\C^{2n},\\Lu\neq 0}}\left(\frac{\|RMLu\|^2}{u^*L^*M^*RMLu}\right)^2+\frac{\|[I_n~0] u\|^2}{\|MLu\|^2} + \left(\frac{\|[0~I_n] u\|^2}{u^*L^*M^*[0~I_n] u}\right)^2: \begin{matrix} u^*[I_n~0]^TMLu \in i\R,\\ ~u^*L^*M^*[0~I_n]u>0 \end{matrix} \right\}  \Bigg\},
        \end{align}
        where $\Omega$ is the set of eigenvectors corresponding to the eigenvalues of the pair $(E,J)$, $\Lambda$ is defined in~\eqref{set:Lambda}, and $L:=[I_n~-i\omega I_n]$,
        \item for  perturbations from $\mathcal S_i$, the distance to a dH pencil with purely imaginary eigenvalues
        is bounded by
        \begin{equation}\label{Si:dist-imagEJR}
            (d_{im}^{\mathcal S_i}(E,J,R))^2\geq \min\left\{ \inf_{x\in \Omega}\frac{\|Rx\|^2}{\|x\|^2}, \inf_{\omega\in \R\setminus \Lambda} \left(\min_{t_1,t_2 \in \R} \lambda_{\max}(G+t_1H_1+t_2H_2) \right)^{-1} \right\},
        \end{equation}
        where $\Omega$ is the set of eigenvectors corresponding to the eigenvalues of the pair $(E,J)$, $\Lambda$ is defined in~\eqref{set:Lambda}, and $G$, $H_1$, and $H_2$ are defined in~\eqref{mat:GH1H2}. 
    \end{enumerate}
\end{theorem}
\begin{proof}
    For $\mathbb S\in \{\mathcal S_d,\mathcal S_i\}$,  the distance to a dHDAE with eigenvalues on the imaginary axis is defined by
    \begin{equation*}
        d_{im}^{\mathbb S}(E,J,R) = \inf\{\nnrm{(\Delta_E,\Delta_J,\Delta_R)}~:~(\Delta_E,\Delta_J,\Delta_R)\in \mathbb S,~\Lambda(E+\Delta_E,J+\Delta_J-R-\Delta_R)\cap i\mathbb R \neq \emptyset \}.
    \end{equation*}
    Using Lemma~\ref{lem:tfae-map}, we obtain
    \begin{align*}
        &d_{im}^{\mathbb S}(E,J,R)\nonumber\\
        &= \inf\{\nnrm{(\Delta_E,\Delta_J,\Delta_R)}:(\Delta_E,\Delta_J,\Delta_R)\in \mathbb S, (R+\Delta_R)x=0 \text{ for some eigenvector $x$ of } (E+\Delta_E,J+\Delta_J) \}\nonumber\\
		&=\inf\{ \nnrm{(\Delta_E,\Delta_J,\Delta_R)}:(\Delta_E,\Delta_J,\Delta_R)\in \mathbb S, (R+\Delta_R)x=0 \text{ for some $x\in \C^n\setminus\{0\}$ satisfying } \nonumber\\ &\hspace{5cm}  (J+\Delta_J)x= i\omega (E+\Delta_E) x, \omega\in \R\}\nonumber\\
		&= \inf_{\omega\in \R}\inf_{x\in\C^{n}\setminus \{0\}} \Big(\inf \{ \nnrm{(\Delta_E,\Delta_J,\Delta_R)}:(\Delta_E,\Delta_J,\Delta_R)\in \mathbb S, (R+\Delta_R)x=0, (J+\Delta_J)x= i\omega (E+\Delta_E) x\} \Big)\nonumber\\
		&=\inf_{\omega \in \R } \rho_\omega^{\mathbb S},
    \end{align*}
    %
    %
    where, for a given scalar $\omega \in \R$, we have
    \begin{align}\label{eq:rho}
        \rho_\omega^{\mathbb S}:=\inf_{x\in \C^{n}\setminus \{0\}}\Big(\inf\{\nnrm{(\Delta_E,\Delta_J,\Delta_R)}^2:(\Delta_E,\Delta_J,\Delta_R)\in \mathbb S,(R+\Delta_R)x=0, (J+\Delta_J)x= i\omega (E+\Delta_E) x\}\Big).
    \end{align}
    By using the set $\Lambda$ of eigenvalues of the pair $(E,J)$ as defined in~\eqref{set:Lambda}, we obtain
    \begin{equation}\label{eq:dist-im1}
        d_{im}^{\mathbb S}(E,J,R)=\inf_{\omega\in \Lambda \cup \R\setminus\Lambda} \rho_{\omega}^{\mathbb S}.
    \end{equation}
    We first consider the case when $\mathbb S=\mathcal S_d(E,J,R)$. Let $\omega\in \R\setminus\Lambda$, then, using Lemma~\ref{lem:tfae-map} in~\eqref{eq:rho}, we obtain
    \begin{align}
        \rho_\omega^{\mathcal S_d}\Big|_{\R\setminus \Lambda}&= \inf_{x\in \C^{n}\setminus \{0\}}\Big(\inf\{\nnrm{(\Delta_E,\Delta_J,\Delta_R)}^2:(\Delta_E,\Delta_J,\Delta_R)\in S_d,\exists~v_E,v_J\in \C^{n}, \text{ s.t. } v:=-i\omega v_E+v_J\neq 0,\nonumber\\
        & \hspace{5.5cm}\Delta_RM v= -RM v,~\Delta_EM v=v_E,~\Delta_JM v=v_J\} \Big).\label{eq:rho1}
    \end{align}
    To proceed, we solve the associated mapping problems under the imposed structure on the set $\mathcal S_d(E,J,R)$. By~{Theorem~\ref{map:def}}, there exist matrices $\Delta_E\leq 0,$ $\Delta_R\leq 0$ and by Theorem~\ref{map:herm} there exists skew-Hermitian $\Delta_J$ satisfying the constraints in \eqref{eq:rho1} if and only if
    \begin{equation*}
        v^*M^*RM v>0,~~ v_E^*M v>0,~ \text{  and }~ v_J^*M v \in i\R.
    \end{equation*}
    Moreover, among all admissible mappings, the minimal norms are given by
    \begin{equation*}
        \|\Delta_R\|=\frac{\|RMv\|^2}{v^*M^*RM v},~\|\Delta_E\|=\frac{\|v_E\|^2}{v_E^*M v},~\text{ and }~\|\Delta_J\|=\frac{\|v_J\|}{\|M v\|}.
    \end{equation*}
    Introducing $u:=[v_J^T,~v_E^T]^T$, $L:=[I_n,~-i\omega I_n]$, and substituting the minimal norm mappings in~\eqref{eq:rho1}, we obtain
    \begin{align}
        \rho_\omega^{\mathcal S_d}\Big|_{\R\setminus \Lambda}\geq {\inf_{\substack{u\in\C^{2n},\\Lu\neq 0}}\Bigg\{\left(\frac{\|RMLu\|^2}{u^*L^*M^*RMLu}\right)^2+\frac{\|[I_n~0] u\|^2}{\|MLu\|^2} + \left(\frac{\|[0~I_n] u\|^2}{u^*L^*M^*[0~I_n] u}\right)^2: \begin{matrix} u^*[I_n~0]^TMLu \in i\R,\\ ~u^*L^*M^*[0~I_n]u>0 \end{matrix}\Bigg\}}.
        \label{rhobound}
    \end{align}
    The inequality in \eqref{rhobound} arises because, in solving the individual minimal-norm mapping problems, we did not explicitly enforce the additional conditions $E+\Delta_E\geq 0$ and $R+\Delta_R\geq 0$.

    We now turn to the case $\omega \in \Lambda$. In this situation, the minimal norm $\nnrm{(\Delta_E,\Delta_J,\Delta_R)}$ is attained by choosing $\Delta_E=0$ and $\Delta_J=0$, and by selecting $\Delta_R$ such that $\Delta_R x_\omega = - R x_\omega$, where $x_\omega$ is an eigenvector of the matrix pair $(E,J)$ corresponding to the eigenvalue $i\omega$.
   Note that $x_\omega^*Rx_\omega \neq 0$. Indeed, if $x_\omega^*Rx_\omega = 0$ then $Rx_\omega=0$ as $R \geq 0$. By Theorem~\ref{thm:singindspec} this implies that $i\omega$ is an eigenvalue of the dHDAE $(E,J-R)$ on the imaginary axis, which is a contradiction as the pair $(E,J-R)$ is assumed to be robustly asymptotically stable.

    Among all Hermitian negative semidefinite perturbations $\Delta_R \leq 0$ satisfying this constraint, the minimal norm is given by
    \[
    \|\Delta_R\| = \frac{\|R x_\omega\|^2}{x_\omega^* R x_\omega},
    \]
    and this minimum is attained for
    \[
    \Delta_R = -\frac{(R x_\omega)(R x_\omega)^*}{x_\omega^* R x_\omega}.
    \]
    Note that, in view of~\cite[Lemma 4.1]{MehMS16}, this choice of $\Delta_R$ also satisfies $R+\Delta_R \geq 0$. Substituting this expression into~\eqref{eq:rho}, we obtain
    \begin{equation}\label{eq:temp11}
    \rho_\omega^{\mathcal S_d}\big|_{\Lambda} = \inf_{\omega \in \Lambda} \frac{{\|R x_\omega\|}^4}{{(x_\omega^* R x_\omega)}^2},
    \end{equation}
    where $x_\omega$ denotes eigenvector associated with the eigenvalue $\omega$. Combining the cases $\omega\in\mathbb{R}\setminus\Lambda$ in~\eqref{rhobound} and $\omega\in\Lambda$ in~\eqref{eq:temp11}, we obtain the bound in~\eqref{Sd:dist-imagEJR}.\\

    We now consider the case $\mathbb S = \mathcal S_i(E,J,R)$. For this structured perturbation set, it follows from~\eqref{eq:rho1} that, for $\omega \in \mathbb{R}\setminus\Lambda$,
    \begin{align}
        \rho_\omega^{\mathcal S_i}\Big|_{\R\setminus \Lambda}&= \inf_{x\in \C^{n}\setminus \{0\}}\Big(\inf\{\nnrm{(\Delta_E,\Delta_J,\Delta_R)}^2~:~(\Delta_E,\Delta_J,\Delta_R)\in \mathcal S_i,~\exists~v_E,v_J\in \C^{n}, \text{ s.t. } v:=-i\omega v_E+v_J\neq 0,\nonumber\\
        & \hspace{5.5cm}\Delta_RM v= -RM v,~\Delta_EM v=v_E,~\Delta_JM v=v_J\} \Big).\label{eq:rho3}
    \end{align}
    We proceed analogously to the case $\mathbb S=\mathcal S_d(E,J,R)$ by solving the associated minimal-norm mapping problems
\begin{equation}\label{eq:temp22}
        \Delta_R M v = - R M v, \quad
    \Delta_E M v = v_E, \quad
    \Delta_J M v = v_J,
    \end{equation}
    with $(\Delta_E,\Delta_J,\Delta_R)\in \mathcal S_i(E,J,R)$. By Theorem~\ref{map:herm}, there exist Hermitian matrices $\Delta_R,\Delta_E$ and a skew-Hermitian matrix $\Delta_J$ 
    satisfying the constraints in~\eqref{eq:temp22} if and only if
    \begin{equation*}
        v^* M^* R M v \in \mathbb{R}, \qquad v_E^* M v \in \mathbb{R}, \qquad v_J^* M v \in i\mathbb{R}.
    \end{equation*}
    The first condition holds trivially since $R$ is Hermitian. The remaining two conditions can equivalently be written as
    \begin{equation*}
        u^* \tilde H_1 u = 0, \quad u^* \tilde H_2 u = 0,
    \end{equation*}
    where $u := [\, v_J^T \;\; v_E^T \,]^T$ and the matrices $\tilde H_1$ and
    $\tilde H_2$ are defined in~\eqref{mat:tildeH}.

    Among all feasible mappings, the minimal norms are given by
    \begin{equation*}
        \|\Delta_R\| = \frac{\|R M v\|}{\|M v\|}, \quad
        \|\Delta_E\| = \frac{\|v_E\|}{\|M v\|}, \quad
        \|\Delta_J\| = \frac{\|v_J\|}{\|M v\|}.
    \end{equation*}
    Hence,
    \begin{equation*}
        \nnrm{(\Delta_E,\Delta_J,\Delta_R)}^2 = \frac{\|v_E\|^2 + \|v_J\|^2 + \|R M v\|^2}{\|M v\|^2} = \frac{u^* G_1 u}{u^* G_2 u},
    \end{equation*}
    where the matrices $G_1$ and $G_2$ are defined in~\eqref{mat:G1G2}. Using these formulas in~\eqref{eq:rho3}, we obtain
    \begin{equation}
        \rho_\omega^{\mathcal S_i}\Big|_{\R\setminus \Lambda} \geq \inf_{u\in \C^{2n}\setminus \{0\}} \left\{ \frac{u^* G_1 u}{u^* G_2 u}~:~ u^*G_2u\neq 0,~ u^*\tilde H_1u=0,~u^*\tilde H_2u=0 \right\}.\label{rhobound1}
    \end{equation}
    Note that we have an inequality in \eqref{rhobound1} as we did not impose the conditions $R+\Delta_R\geq0, E+\Delta_E\geq 0$ from the set $\mathcal S_i(E,J,R)$ while solving the mapping problems. Finally, since $G_1$ is positive definite, we may introduce its principal positive definite square root and obtain the equivalent formulation
    \begin{equation}\label{eq:temp1}
        \rho_\omega^{\mathcal S_i}\Big|_{\R\setminus \Lambda} \geq \left( \sup_{y\in \C^{2n}\setminus \{0\}} \left\{ \frac{y^* G y}{y^*y}~:~ y^*H_1y=0,~y^*H_2y=0 \right\}\right)^{-1},
    \end{equation}
    where $y:=G_1^{1/2}u$, and the matrices $G:=G_1^{-1/2}G_2G_1^{-1/2},H_1:=G_1^{-1/2}\tilde H_1G_1^{-1/2},$ and $H_2:=G_1^{-1/2}\tilde H_2G_1^{-1/2}$ are as in~\eqref{mat:GH1H2}. Note that in~\eqref{eq:temp1} we omitted the condition $u^*G_2u\neq 0$ (or equivalently $y^*Gy\neq 0$), as the distance to a pencil with purely imaginary eigenvalues $d_{im}^{\mathcal S_i}(E,J,R)$ is finite, therefore including vectors $u$ such that $u^*G_2u=0$, will not effect the supremum.

Also note that $\alpha H_1+\beta H_2$ is indefinite (i.e., strictly not semidefinite) for every $(\alpha,\beta)\in \R^2 \setminus \{0\}$. Thus in view of~\cite[Theorem 3.2]{BorKMS14}, we have
\begin{equation}\label{eq:temp33}
\rho_\omega^{\mathcal S_i}\Big|_{\R\setminus \Lambda}
\geq \left(\min_{t_1,t_2 \in \R} \lambda_{\max}(G+t_1H_1+t_2H_2) \right)^{-1}.
\end{equation}
By following the arguments similar to the case $\rho_\omega^{\mathcal S_d}\Big|_{ \Lambda}$ and using the minimal norm Hermitian mapping from Theorem~\ref{map:herm}, we obtain
  \begin{equation}\label{eq:temp44}
    \rho_\omega^{\mathcal S_i}\big|_{\Lambda} = \inf_{\omega \in \Lambda} \frac{{\|R x_\omega\|}^2}{{\|x_\omega\|}^2},
    \end{equation}
  where $x_\omega$ is the eigenvector associated with the eigenvalue $\omega$. Combining the cases $\omega\in\mathbb{R}\setminus\Lambda$ in~\eqref{eq:temp33} and $\omega\in\Lambda$ in~\eqref{eq:temp44}, we obtain the bound in~\eqref{Si:dist-imagEJR}.
\end{proof}

\begin{remark}{\rm
    Note that Theorem~\ref{thm:dist-imagEJR} provides only a lower bound for the smallest perturbation that moves an eigenvalue to the imaginary axis. This is because, in solving the associated minimal-norm mapping problems in the proof, we did not enforce the additional constraints in the perturbation set $\mathbb S \in \{\mathcal S_d(E,J,R),\mathcal S_i(E,J,R)\}$, namely the conditions $E+\Delta_E \ge 0$ and $R+\Delta_R \ge 0$.

    If the infimum in~\eqref{Sd:dist-imagEJR} and~\eqref{Si:dist-imagEJR} is attained at some $(\hat \omega,\hat x, \hat u)$ and $(\hat \omega, \hat x, \hat t_0,\hat t_1)$ respectively, then the corresponding optimal perturbations $(\hat \Delta_E$, $\hat \Delta_J, \hat \Delta_R)$ can be constructed explicitly as described in the proof of Theorem~\ref{thm:dist-imagEJR}.
    %
    %
    Moreover, if these perturbations additionally satisfy $E+\hat \Delta_E \ge 0$ and $R+\hat \Delta_R \ge 0$, then the lower bound in~\eqref{Sd:dist-imagEJR} and~\eqref{Si:dist-imagEJR} are tight, and equality holds.
    }
\end{remark}
After characterizing the distance to the nearest dHDAE with purely imaginary eigenvalues, in the next subsection we treat the distance to the nearest singular dHDAE.
\subsection{Distance to the nearest singular dHDAE}\label{sec:distsing}

The unstructured distance to singularity $d_{sing}(E,J,R):=d_{sing}^{\mathcal S_d}{\mathbb(E,J,R)}$, when $\mathbb S={(\C^{n,n})}^3$ in~\eqref{def:dist-sing} and the structured distance $d_{sing}^{\mathcal S_d}(E,J,R)$ were studied in~\cite{PraS22}. For the sake of completeness, we state the following result from~\cite{PraS22} in our notation.

\begin{theorem}\label{thm:dist-singEJR-d}
        Consider an asymptotically stable dHDAE system of the form~\eqref{dhdae1}. Then
    \begin{enumerate}
        \item the unstructured distance to singularity $d_{sing}(E,J,R)$ is given by
    \begin{equation}\label{Si:dist-sing-EJR1}
        d_{sing}(E,J,R)=\sqrt{\lambda_{\min}(E^2-J^2+R^2)}.
    \end{equation}
        \item the distance to singularity $d_{sing}^{\mathcal S_d}(E,J,R)$ with respect to perturbations from the set $\mathcal S_d(E,J,R)$ is given by
\begin{equation}\label{Sd:dist-sing-EJR}
(d_{sing}^{\mathcal S_d}(E,J,R))^2=\min\left\{d_{sing}^{\mathcal S_d}(J,R)^2,d_{sing}^{\mathcal S_d}(E,J)^2,
\min_{\alpha \in \mathcal M_3} \frac{\alpha^*J^*J\alpha}{\alpha^*\alpha}+
 \left(\frac{\alpha^*R^2\alpha}{\alpha^*R\alpha}\right)^2+ \left(\frac{\alpha^*E^2\alpha}{\alpha^*E\alpha}\right)^2
\right\},
   \end{equation}
where  $\mathcal M_3=\text{\rm ker}(R)^{c}\cap \text{\rm ker}(E)^{c}$,
and $d_{sing}^{\mathcal S_d}(J,R)$ and
$d_{sing}^{\mathcal S_d}(E,J)$ are given by~\cite[Table A.1]{PraS22}.
    \end{enumerate}
\end{theorem}

\begin{theorem}\label{thm:dist_singEJR_i}
    Consider an asymptotically stable dHDAE system of the form~\eqref{dhdae1}. Then
the distance to singularity $d_{sing}^{\mathcal S_i}(E,J,R)$ with respect to perturbations from the set $\mathcal S_i(E,J,R)$ is given by
    %
    \begin{equation}\label{Si:dist-sing-EJR2}
   d_{sing}^{\mathcal S_i}(E,J,R)= \sqrt{\lambda_{\min}(E^2-J^2+R^2)}= d_{sing}(E,J,R).
    \end{equation}
\end{theorem}
\begin{proof}
In view of~\eqref{def:dist-sing} and Corollary~\ref{cor}, we have
	\begin{align}
		&d_{sing}^{\mathcal S_i}(E,J,R)\nonumber\\
		&=\inf \left\{ \nnrm{(\Delta_E,\Delta_J,\Delta_R)}:(\Delta_E,\Delta_J,\Delta_R)\in \mathcal S_i,~ \ker(E+\Delta_E) \cap \ker(J + \Delta_J) \cap \ker(R+\Delta_R) \neq \{0\} \right\}\nonumber\\
		&=\inf \{ \nnrm{(\Delta_E,\Delta_J,\Delta_R)}:(\Delta_E,\Delta_J,\Delta_R)\in \mathcal S_i,~x\in \C^{n}\setminus\{0\},~ (E+\Delta_E)x=0, (J + \Delta_J)x=0,\nonumber \\
        &\hspace{10cm}(R+\Delta_R)x=0 \}.\label{Sd:dist-sing-EJR1}
	\end{align}
For a given $x\in \C^{n}\setminus \{0\}$, to obtain the minimal norm perturbation $(\Delta_E,\Delta_J,\Delta_R) \in \mathcal S_i(E,J,R)$ we have to solve the mapping problem
	\begin{equation}\label{mapping:constraints}
		(E+\Delta_E)x=0,~ \Delta_Rx=-Rx,~\text{and }~ \Delta_Jx=-Jx.
	\end{equation}
    By Theorem~\ref{map:herm}, there exist Hermitian matrices $\Delta_R,\Delta_E$ and a skew-Hermitian matrix $\Delta_J$ satisfying these constraints if and only if
    \begin{equation*}
        x^*R x \in \mathbb{R}, \qquad x^* E x \in \mathbb{R}, \qquad x^* J x \in i\mathbb{R},
    \end{equation*}
    which holds trivially due to the structure of the matrices $E,J$ and $R$. Among all feasible mappings, the minimal norms are given by
    \begin{equation*}
        \|\Delta_R\| = \frac{\|R x\|}{\|x\|}, \quad
        \|\Delta_E\| = \frac{\|Ex\|}{\|x\|}, \quad
        \|\Delta_J\| = \frac{\|Jx\|}{\|x\|}.
    \end{equation*}
    and are attained by $\Delta_E=\hat H_{(x,-Ex)}, \Delta_J=-i\hat H_{(x,-iJx)},$ and $\Delta_R=\hat H_{(x,-Rx)}$, where $\hat H$ is defined in~\eqref{def:hatH}. This implies that
    \begin{equation}\label{norm:singSi}
        \nnrm{(\Delta_E,\Delta_J,\Delta_R)}^2 = \frac{x^*E^*E x}{x^*x}+\frac{x^*J^*J x}{x^*x}+\frac{x^*R^*R x}{x^*x} = \frac{x^*(E^2-J^2+R^2) x}{x^*x}.
    \end{equation}
    Using~\eqref{norm:singSi} and $S_i(E,J,R) \subseteq \{(\Delta_E,\Delta_J,\Delta_R)~:~ \Delta_E^*=\Delta_E,\Delta_J^*=-\Delta_J,\Delta_R^*=\Delta_R\}$, we obtain
    \begin{equation}\label{Si:dist-sing-EJR3}
        (d_{sing}^{\mathcal S_i}(E,J,R))^2\geq \inf_{x\in \C^{n}\setminus\{0\}} \frac{x^*(E^2-J^2+R^2)x}{x^*x} = \lambda_{\min}(E^2-J^2+R^2).
    \end{equation}
    This proves the inequality in~\eqref{Si:dist-sing-EJR2}. To prove the equality we consider two cases.

    Case 1:  $E>0,~R>0$. Let $\hat x$ be the eigenvector corresponding to the smallest eigenvalue of $E^2-J^2+R^2$. Using Lemma~\ref{map:herm}, there exists optimal $(\hat \Delta_E,\hat \Delta_J,\hat \Delta_R)$ such that $\hat \Delta_E^*=\hat\Delta_E,\hat\Delta_J^*=-\hat\Delta_J,\hat\Delta_R^*=\hat\Delta_R$ satisfying
    \begin{align}
        \hat \Delta_E\hat x&=-E\hat x,~ \hat \Delta_J\hat x=-J\hat x,~ \text{ and }~ \hat \Delta_R \hat x = -R\hat x\label{norm:singSitemp} \\
        \nnrm{(\hat\Delta_E,\hat\Delta_J,\hat\Delta_R)}^2 &= \frac{\hat x^*(E^*E+J^*J+R^*R \hat x}{\hat x^*\hat x} =\lambda_{\min}(E^2-J^2+R^2).\label{norm:singSi1}
    \end{align}
    Thus, if we show that $E+\hat \Delta_E\geq 0$ and $R+\hat \Delta_R\geq 0$, then this will imply that $(\hat \Delta_E,\hat \Delta_J,\hat \Delta_R)\in S_i(E,J,R)$, and thus in view of~\eqref{norm:singSi1} we have equality in~\eqref{Si:dist-sing-EJR3}. Note that the optimal matrix $\hat H$ attaining the minimal norm in Lemma~\ref{map:herm} has at most one negative eigenvalue, since it is either a rank one or rank two matrix, and if it has rank two, then it is easy to check that $y\pm \frac{\|y\|}{\|x\|}x$ are eigenvectors corresponding to the eigenvalues $\pm \frac{\|y\|}{\|x\|}$, respectively. This implies that $\hat\Delta_E$ and $\hat \Delta_R$ has at most one negative eigenvalue. By using~\eqref{norm:singSitemp}, we obtain
    \begin{align*}
        (E+\hat \Delta_E)\hat x= E\hat x + \hat \Delta_E \hat x= E\hat x - E\hat x=0,\\
        (R+\hat \Delta_R)\hat x= R\hat x + \hat \Delta_R \hat x= R\hat x - R\hat x=0.
    \end{align*}
    Thus, $\hat\Delta_E$ and $\hat \Delta_R$ are matrices with at most one negative eigenvalue and also satisfy that $E+\hat \Delta_E$ and $R+\hat \Delta_R)$ are singular.  By~\cite[Lemma 4.4]{MehMS16} this implies that $E+\hat \Delta_E\geq 0$ and $R+\hat \Delta_R \geq 0$.

    Case 2: $E\geq 0,~R\geq 0$. In this case, for any $\epsilon > 0$, $E_\epsilon=E+\epsilon I, R_\epsilon=R+\epsilon I$ are positive definite and by following the arguments of Case 1, we have
    \[
    (d_{\text{sing}}^{S_i}(E_\epsilon, J,R_\epsilon))^2= \lambda_{\min}(E_\epsilon^2-J^2+R_\epsilon^2).
    \]
    Since eigenvalues depend continuously on matrix entries, the function $\lambda_{\text{min}}(E_\epsilon^2-J^2+R_\epsilon^2)$ is continuous in $\epsilon$. Thus, we have
    \begin{eqnarray*}
        (d_{\text{sing}}^{S_i}(E,J,R))^2=\lim_{\epsilon \to 0} (d_{\text{sing}}^{S_i}(E_\epsilon,J,R_\epsilon))^2=
        \lim_{\epsilon \to 0} \lambda_{\text{min}}(E_\epsilon^2-J^2+R_\epsilon^2) = \lambda_{\min}(E^2-J^2+R^2).
    \end{eqnarray*}
This completes the proof.
\end{proof}

\subsection{Distance to the nearest dHDAE with higher index}\label{sec:disthi}
Having established an explicit characterization of the distance to singularity for dHDAE systems,
we now turn our attention to the nearest system of higher index
under the same classes of structure-preserving perturbations, $ {\mathcal S_d}(E,J,R)$ and ${\mathcal S_i}(E,J,R).$

\begin{theorem}\label{thm:dist-hiEJR}
    Consider a robustly asymptotically stable dHDAE system of the form~\eqref{dhdae1}. For each $1\leq k \leq n$, let $N_k$ be a matrix with orthonormal columns spanning a basis of $\ker(E+\hat \Delta_E)$, where $\hat \Delta_E$ is the minimal norm (optimal) perturbation such that $\rank (E+\hat \Delta_E)=k$. Further, let $\lambda_{n-k+1}(E)$ be the $(n-k+1)$st smallest eigenvalue of $E$.
    \begin{enumerate}
    \item If the infimum
    \[
    \rho({N_k}):= \inf_{x \neq 0} \left\{
    \frac{\|N_k^* J N_k x\|^2}{\|x\|^2} +
    \frac{\|N_k^* R N_k x\|^4}{\big(x^* N_k^* R N_k x\big)^2} \right\}
    \]
    is attained at $\hat x \neq 0$ such that $R+\hat \Delta_R \geq 0$ with $\hat \Delta_R=-\frac{(N_k N_k^* R N_k \hat x)(N_k N_k^* R N_k \hat x)^*}{{\hat x}^* N_k^* R N_k \hat x}$, then the distance to higher index $d_{hi}^{\mathcal S_d}(E,J,R)$ with respect to perturbations from the set $\mathcal S_{d}(E,J,R)$ is bounded as
    \begin{equation}
            (d_{hi}^{\mathcal S_d}(E,J,R))^2 \leq \min_{1\leq k \leq n}\left\{ \lambda_{n-k+1}(E)^2 + \rho(N_k) \right\}.
\end{equation}
\item the distance to higher index $d_{hi}^{\mathcal S_i}(E,J,R)$ with respect to perturbations from the set $\mathcal S_{i}(E,J,R)$ is bounded as
        %

\begin{equation}
            (d_{hi}^{\mathcal S_i}(E,J,R))^2 \leq \min_{1\leq k \leq n}\left\{ \lambda_{n-k+1}(E)^2 + \sigma_{\min}\left(\mat{c}N_k^*RN_k\\ N_k^*JN_k\rix \right)^2 \right\},
\end{equation}
       provided that $R+N_k\hat H_{(x,\,-N_k^* R N_k x)}N_k^* \geq 0$, where $x$ is the right singular vector corresponding to the smallest singular value of $\mat{c}N_k^*RN_k\\ N_k^*JN_k\rix$, and $\hat H$ is defined by~\eqref{def:hatH}.
    \end{enumerate}
\end{theorem}
\begin{proof}
    Using Definition~\eqref{def:distances}, the distance to higher index under structured perturbations from the set $\mathbb S\in \{\mathcal S_d,\mathcal S_i\}$ is given by
    \[
    d_{hi}^{\mathbb S}(E,J,R) = \inf\Big\{\nnrm{(\Delta_E,\Delta_J,\Delta_R)} ~:~ (\Delta_E,\Delta_J,\Delta_R)\in \mathbb S,~ \text{ index}(E+\Delta_E,J+\Delta_J-R-\Delta_R) > 1 \Big\}.
    \]
    In view of Theorem~\ref{thm:stabledh}, this can be equivalently written as
    \begin{align*}
        &d_{hi}^{\mathbb S}(E,J,R) \\ &= \inf\Big\{ \nnrm{(\Delta_E,\Delta_J,\Delta_R)} ~:~ (\Delta_E,\Delta_J,\Delta_R)\in \mathbb S,~ \text{rank}\big(N(E+\Delta_E)^* (J+\Delta_J-R-\Delta_R) N(E+\Delta_E) \big) < n \Big\},
    \end{align*}
    where $N(E+\Delta_E)$ denotes a basis matrix for the null space of $E+\Delta_E$.

    Observe that the constraint in this formulation involves second-order perturbation terms through products of $\Delta_E$ with $\Delta_J$ and $\Delta_R$. To address this difficulty, we adopt a two-stage approach. First, we determine an optimal perturbation $\Delta_E$ such that $\dim \ker(E+\Delta_E) = k$. Then, for this fixed $\Delta_E$, we determine the minimal structured perturbations $(\Delta_J,\Delta_R)$ that satisfy the rank condition. Let $E = U \Sigma U^*$ be the spectral decomposition of $E$, with $\Sigma = \diag \{\lambda_1,\lambda_2,\ldots,\lambda_n\}$, where $\lambda_1 \ge \cdots \ge \lambda_n$. The optimal perturbation $\Delta_E$ achieving ${\rank}(E+\Delta_E)=n-k$ is given by
    \[
    \hat \Delta_E = - U \diag\{0,0,\ldots,0,\lambda_{n-k+1},\lambda_{n-k+2}, \ldots,\lambda_{n-1},\lambda_{n}\}U^*.
    \]
    Let $N_k$ be a  matrix whose columns form a basis for $\ker(E+\hat \Delta_E)$.
    Then we obtain
    \begin{align}\label{Sd:muintro}
        (d_{hi}^{\mathbb S}(E,J,R))^2 &\leq  \inf\Big\{ \lambda_{n-k+1}(E)^2 + \|\Delta_J\|^2 + \|\Delta_R\|^2 : (\hat \Delta_E,\Delta_j,\Delta_R)\in \mathbb S,\nonumber \\ & \hspace{3cm} \exists\, x\neq 0 \text{ such that } N_k^* (J+\Delta_J-R-\Delta_R) N_k x = 0 \Big\}
        \nonumber\\
        &=\lambda_{n-k+1}(E)^2 + \mu_k^{\mathbb S},
    \end{align}
    where
    \begin{eqnarray}\label{hu_hi}
        \mu_k^{\mathbb S}:=\inf\Big\{ \|\Delta_J\|^2 + \|\Delta_R\|^2 : (\hat \Delta_E,\Delta_j,\Delta_R)\in \mathbb S,\exists\, x\neq 0 \text{ such that } N_k^* (J+\Delta_J-R-\Delta_R) N_k x = 0 \Big\}.
    \end{eqnarray}
    Since $J+\Delta_J$ is skew-Hermitian and $R+\Delta_R$ is Hermitian positive semidefinite, the condition
    \[
    N_k^*(J+\Delta_J - R - \Delta_R)N_kx = 0
    \]
    holds if and only if
    \[
    N_k^*(J+\Delta_J)N_kx = 0 \quad \text{and} \quad N_k^*(R+\Delta_R)N_kx = 0.
    \]
    Therefore,
    \begin{eqnarray}\label{eq:temp2}
        \mu_k^{\mathbb S} &= & \inf\Big\{  \|\Delta_J\|^2 + \|\Delta_R\|^2 : (\hat \Delta_E,\Delta_j,\Delta_R)\in \mathbb S, \exists\, x\neq 0 \text{ such that } \nonumber
        \\ && \hspace{4.5cm} N_k^* (J+\Delta_J) N_k x = 0,~ N_k^* (R+\Delta_R) N_k x = 0 \Big\}.
    \end{eqnarray}
    %
    We now solve the associated minimal-norm mapping problems.
    First, let $\mathbb S=\mathcal S_d(E,J,R)$ and $R>0$. By Corollary~\ref{map:herm-corr}
    %
    %
    and~Corollary~\ref{map:def-corr}, there exist a skew-Hermitian matrix $\Delta_J$ and a Hermitian negative semidefinite matrix $\Delta_R\le 0$ satisfying the constraints in~\eqref{eq:temp2} if and only if
    \[
    x^* N_k^* J N_k x \in i\mathbb{R}
    \quad \text{and} \quad
    x^* N_k^* R N_k x > 0.
    \]
    These conditions are automatically satisfied due to the structural properties of the matrices $J$ and $R$.
    Among all such admissible mappings, the minimal norms are given by
    \[
    \|\hat \Delta_J\| = \frac{\|N_k^* J N_k x\|}{\|x\|}, \qquad
    \|\hat \Delta_R\| = \frac{\|N_k^* R N_k x\|^2}{x^* N_k^* R N_k x},
    \]
    and are attained by
    \begin{equation}\label{eq:temp55}
    \hat \Delta_J = -i\,N_k\hat H_{(N_k x,-iN_k^* J N_k x)}N_k^* \quad \text{and}\quad
    \hat \Delta_R = -\frac{N_k\big(N_k^* R N_k x\big) \big(N_k^* R N_k x\big)^*N_k^*}{x^* N_k^* R N_k x},
    \end{equation}
    where $\hat H$ is defined in~\eqref{def:hatH}.
    If $R\ge 0$ is singular and $x^* N_k^* R N_k x = 0$, then necessarily $N_k N_k^* R N_k x = 0$. In this case, adopting the convention $0/0 = 0$, the choice $\Delta_R = 0$ satisfies both the mapping constraint in~\eqref{eq:temp2} and the minimal-norm condition. Note that
    \[
	(\hat \Delta_E,\hat \Delta_J,\hat \Delta_R)\in
	\{(\Delta_E,\Delta_J,\Delta_R)~:~ \Delta_E^*\leq 0, E+\Delta_E\leq 0,\ \Delta_J^*=-\Delta_J, \Delta_R\leq 0\} \supseteq \mathcal S_d(E,J,R).
	\]

        Combining these observations in~\eqref{eq:temp2}, we obtain
    \begin{equation}\label{Sd:high11}
    \mu_k^{\mathcal S_d} \geq  \inf_{x \neq 0} \left\{
    \frac{\|N_k^* J N_k x\|^2}{\|x\|^2} +
    \frac{\|N_k^* R N_k x\|^4}{\big(x^* N_k^* R N_k x\big)^2} \right\}.
    \end{equation}
    Thus, if $R+\hat \Delta_R \geq 0$, then this will imply that $(\hat \Delta_E,\hat \Delta_J,\hat \Delta_R) \in \mathcal S_d(E,J,R)$, and hence we have equality in~\eqref{Sd:high11},
       and thus from~\eqref{Sd:muintro} and~\eqref{Sd:high11}, we have
    \[
    (d_{hi}^{\mathcal S_d}(E,J,R))^2 \le \min_{1 \le k \le n} \inf_{x \neq 0} \left\{ \lambda_{n-k+1}(E)^2 +
    \frac{\|N_k^* J N_k x\|^2}{\|x\|^2} +
    \frac{\|N_k^* R N_k x\|^4}{\big(x^* N_k^* R N_k x\big)^2} \right\}.
    \]
    Next, consider $\mathbb S=\mathcal S_i(E,J,R)$. By Theorem~\ref{map:herm}, there exist a skew-Hermitian matrix $\Delta_J$ and a Hermitian matrix $\Delta_R$ satisfying the constraints in~\eqref{eq:temp2} if and only if
    \[
    x^* N_k^* J N_k x \in i\mathbb{R}
    \quad \text{and} \quad
    x^* N_k^* R N_k x \in \mathbb{R},
    \]
    which again holds trivially due to the structure of $J$ and $R$.
    The corresponding minimal norms are
    \[
    \|\Delta_J\| = \frac{\|N_k^* J N_k x\|}{\|x\|},
    \qquad
    \|\Delta_R\| = \frac{\|N_k^* R N_k x\|}{\|x\|},
    \]
    and are attained by
    \[
    \Delta_J = -iN_k\,\hat H_{(x,\, -i N_k^* J N_k x)}N_k^*,
    \qquad
    \Delta_R = N_k\hat H_{(x,\,-N_k^* R N_k x)}N_k^*,
    \]
    where $\hat H$ is defined in~\eqref{def:hatH}.
    %
    Using these minimal norm mappings and $\{(\hat \Delta_E,\Delta_J,\Delta_R)~:~\Delta_J^*=-\Delta_J, \Delta_R^*=\Delta_R\} \subseteq S_i(E,J,R)$, we obtain
    \begin{equation}\label{Sd:high22}
    \mu_k^{\mathcal S_i} \geq  \inf_{x \neq 0} \left\{
    \frac{\|N_k^* J N_k x\|^2}{\|x\|^2} +
    \frac{\|N_k^* R N_k x\|^2}{\|x\|^2} \right\} = \sigma_{\min}\left(\mat{c}N_k^*RN_k\\ N_k^*JN_k\rix \right)^2 = \sigma_{\min}(A_k)^2.
    \end{equation}
    To prove the equality in~\eqref{Sd:high22}, we consider two cases.

    Case 1: $R>0$. Let $\hat x$ be a right singular vector corresponding to the smallest singular value of $A_k$. Using Lemma~\ref{map:herm}, there exists optimal
    $(\hat \Delta_J,\hat \Delta_R)$ such that $\hat \Delta_J^*=-\hat\Delta_J$, $\hat\Delta_R^*=\hat\Delta_R$ satisfying
       \begin{align}
       N_k^*\hat\Delta_JN_k\hat x&=-N_k^*JN_k\hat x,  \quad
    N_k^*\hat\Delta_RN_k\hat x=-N_k^*RN_k\hat x \label{norm:idefhigh1} \\
        \nnrm{(\hat \Delta_J,\hat \Delta_R)}^2&=\frac{\|N_k^*JN_k\hat x\|^2}{\|\hat x\|^2}+ \frac{\|N_k^*RN_k \hat x\|^2}{\|\hat x\|^2}=\sigma_{\min}\left(\mat{c}N_k^*RN_k\\ N_k^*JN_k\rix \right)^2.
        \label{norm:idefhigh2}
    \end{align}
    Thus, if $R+\hat \Delta_R \geq 0$, then this will imply that $(\hat \Delta_E,\hat \Delta_J,\hat \Delta_R) \in S_i(E,J,R)$, and thus we have equality in~\eqref{Sd:high22}.

  Case 2: If $R\geq 0$ is singular, then for any $\epsilon >0$, $R_\epsilon=R+\epsilon I_n$ is positive definite. Let $\mu_{k,\epsilon}^{\mathcal S_i}$ be as defined in~\eqref{hu_hi}, where $R$ is replaced by $R_\epsilon$. Then by following the arguments of Case 1, we have
\begin{equation*}
    \mu_{k,\epsilon}^{\mathcal S_i}  = \sigma_{\min}\left(\mat{c}N_k^*R_\epsilon N_k\\ N_k^*JN_k\rix \right)^2.
    \end{equation*}
This implies that
\begin{equation*}
            (d_{hi}^{\mathcal S_i}(E,J,R_\epsilon))^2 \leq \min_{1\leq k \leq n}\left\{ \lambda_{n-k+1}(E)^2 + \sigma_{\min}\left(\mat{c}N_k^*R_\epsilon N_k\\ N_k^*JN_k\rix \right)^2 \right\}.
\end{equation*}
Since singular values depend continuously on matrix entries, the function $\mu_{k,\epsilon}^{\mathcal S_i} $ is continuous in  $\epsilon$. Thus, we have
\begin{equation*}
            (d_{hi}^{\mathcal S_i}(E,J,R))^2=
            \lim_{\epsilon \to 0} (d_{hi}^{\mathcal S_i}(E,J,R_\epsilon))^2
            \leq \min_{1\leq k \leq n}\left\{ \lambda_{n-k+1}(E)^2 + \sigma_{\min}\left(\mat{c}N_k^*R N_k\\ N_k^*JN_k\rix \right)^2 \right\}.
\end{equation*}
This completes the proof.
\end{proof}

\begin{remark}{\rm
The assumption $R+\hat \Delta_R \geq 0$ with $\hat \Delta_R=N_k \hat H N_k^*$ in estimating $d_{hi}^{\mathcal S_i}(E,J,R)$ in Theorem~\ref{thm:dist-hiEJR} was needed to show that $(\hat \Delta_E,\hat \Delta_J,\hat \Delta_R) \in S_i(E,J,R)$, which indeed gives  equality in~\eqref{Sd:high22}. However, we note that the restriction of the matrix $R+\hat \Delta_R$ to the subspace $N_k$ is always positive semidefinite. In fact, the optimal matrix $\hat H$ attaining the minimal norm in Theorem~\ref{map:herm} has at most one negative eigenvalue, since it is either a rank-one or rank-two matrix, and if it has rank two, then it is easy to check that $y\pm \frac{\|y\|}{\|x\|}x$ are eigenvectors corresponding to the eigenvalues $\pm \frac{\|y\|}{\|x\|}$, respectively. This implies that $\hat\Delta_R$ has at most one negative eigenvalue. Also, the matrix $N_k^*(R+\hat \Delta_R)N_k$ is singular. By \cite[Lemma 4.4]{MehMS16} this implies that $N_k^*(R+\hat \Delta_R)N_k \geq 0$. A similar remark also holds for the assumption $R+\hat \Delta_R \geq 0$ with $\hat \Delta_R=-\frac{(N_k N_k^* R N_k \hat x)(N_k N_k^* R N_k \hat x)^*}{{\hat x}^* N_k^* R N_k \hat x}$ in estimating $d_{hi}^{\mathcal S_d}(E,J,R)$ in Theorem~\ref{thm:dist-hiEJR}.
}
\end{remark}

\begin{remark}{\rm
    Observe that the definition of the distance to higher index $d_{hi}^{\mathbb S}(E,J,R)$ involves second-order perturbation terms, since the rank condition depends on products of perturbations. To overcome this difficulty, in the proof of the theorem we have employed a two-stage  procedure: first perturbing $E$ to increase the dimension of its null space, and then determining the minimal structured perturbations in $J$ and $R$. As a consequence, the result obtained in the above theorem provides only an upper bound for the distance to the nearest dHDAE system of higher index.

    At present, we cannot make a definitive statement about the tightness of this bound; nevertheless, it is expected to provide a good estimate in practice. In contrast, when perturbations are restricted to $J$ and $R$, i.e., in the computation of $d_{hi}^{\mathbb S}(J,R)$, then no second-order perturbation terms arise. This allows for an exact characterization of the distance to the nearest dHDAE of higher index, see the following section.
    }
\end{remark}

In this section we have derived characterizations and bounds for the three distances under structured perturbations to all three coefficients. Since in some applications the matrix $E$ is fixed and not subject to perturbations, in the next section we present perturbation results where the perturbation is restricted  to $\Delta_E=0$.

\section{Distances under  partial perturbations}\label{sec:partial}
In this section, we present analytic results for the various distances defined in Definition~\ref{def:distances} for dHDAE systems of the form~\eqref{dhdae1} by fixing $E$ and allowing only structured perturbations to $J$ and $R$, i.e. by considering the perturbation sets  as $\mathcal S_d(J,R):=\mathcal S_d(0,J,R)$
and $\mathcal S_i(J,R):=\mathcal S_i(0,J,R)$ in~\eqref{set:Sd} and~\eqref{set:Si}, respectively. The corresponding distances are denoted by $d_{sing}^{\mathbb S}(J,R)$, $d_{hi}^{\mathbb S}(J,R)$, $d_{im}^{\mathbb S}(J,R)$, and $d_{inst}^{\mathbb S}(J,R)$,
respectively, where $\mathbb S\in \{\mathcal S_d,\mathcal S_i\}$.


If the matrix $E$ is nonsingular, then the distance to the boundary of the robustly asymptotically stable systems is equal to the distance to a dHDAE with purely imaginary eigenvalues, as the perturbed system will always be regular and of index zero. This distance was obtained in~\cite{BenMPS25} by transforming the dHDAE~\eqref{dhdae1} to an ordinary dH system of the form $\dot{x}=(J-R)Q$, where $Q=E^{-1}$.

The situation is different for $E\geq 0$ with $E$ being singular. In contrast to the case $E>0$ we now can have loss of regularity of the pencil, an increase in the index, or the presence of a finite eigenvalue on the imaginary axis.
In the following, we first focus on moving finite eigenvalues to the imaginary axis, considering only structure-preserving perturbations in $J$ and $R$ while keeping $E$ fixed. The following result provides an explicit characterization of this distance.
\begin{theorem}\label{thm:dist-imJR}
	Consider a robustly asymptotically stable dHDAE system of the form~\eqref{dhdae1}. Then
	\begin{enumerate}
		\item the distance to the nearest dH pencil with purely imaginary eigenvalues $d_{im}^{\mathcal S_d}(J,R)$ is given by
		\begin{equation}\label{eq:dist-imp}
			(d_{im}^{\mathcal S_d}(J,R))^2 = \inf_{\omega \in \R} \inf_{x\in \C^{n}\setminus \{0\}} \left\{\left(\frac{x^*R^2x}{x^*Rx} \right)^2+\frac{x^*(i\omega E -J)^*(i\omega E -J)x}{x^*x} \right\},
		\end{equation}
		\item the distance to the nearest dH pencil with purely imaginary eigenvalues $d_{im}^{\mathcal S_i}(J,R)$ is given by
		\begin{equation}\label{eq:dist-imp1}
			d_{im}^{\mathcal S_i}(J,R)=\inf_{\omega\in \R} \sigma_{\min}\left( \mat{c}R\\i\omega E - J \rix \right).
		\end{equation}
	\end{enumerate}
\end{theorem}
\begin{proof}
	By Definition~\ref{def:distances}, we have
	\begin{align*}
		d_{im}^{\mathbb S}(J,R)=\inf \{ \nnrm{(\Delta_J,\Delta_R)}~:~(\Delta_J,\Delta_R)\in\mathbb S,~\Lambda(E,J+\Delta_J - (R+\Delta_R)) \cap i\R \neq \emptyset \}.
	\end{align*}
	Since for $(\Delta_J,\Delta_R)\in \mathbb S\in \{\mathcal S_d(J,R),\mathcal S_i(J,R)\}$, the perturbed dH pair $(E,J+\Delta_J - (R+\Delta_R))$ still has the dH structure, by using~\cite[Lemma 4.1]{MehMS16}, we obtain
	\begin{align}
		d_{im}^{\mathbb S}(J,R)
		&=\inf \{ \nnrm{(\Delta_J,\Delta_R)}:(\Delta_J,\Delta_R)\in\mathbb S,~(R+\Delta_R)x=0 \text{ for some eigenvector $x$ of } (E,J+\Delta_J) \}\nonumber\\
		&=\inf\{ \nnrm{(\Delta_J,\Delta_R)}:(\Delta_J,\Delta_R)\in\mathbb S,~ (R+\Delta_R)x=0 \text{ for some $x\in \C^n\setminus\{0\}$ satisfying } \nonumber\\ &\hspace{5cm}  (J+\Delta_J)x= i\omega E x, \omega\in \R\}\nonumber\\
		&= \inf_{\omega\in \R}\inf_{x\in\C^{n}\setminus \{0\}} \Big(\inf \{ \nnrm{(\Delta_J,\Delta_R)}:(\Delta_J,\Delta_R)\in\mathbb S,~ (R+\Delta_R)x=0, (J+\Delta_J)x= i\omega E x\} \Big)\nonumber\\
		&=\inf_{\omega \in \R } \vartheta^{\mathbb S}_{\omega} \label{eq:def11},
	\end{align}
	%
	where, for a given scalar $\omega$, we have
	\begin{align}\label{eq:def12}
		\vartheta^{\mathbb S}_{\omega}:= \inf_{x\in\C^{n}\setminus \{0\}} \Big(\inf \{ \nnrm{(\Delta_J,\Delta_R)}:(\Delta_J,\Delta_R)\in\mathbb S,~ (R+\Delta_R)x=0,(J+\Delta_J)x= i\omega E x\} \Big).
	\end{align}
	Let $\mathbb S=\mathcal S_d(J,R)$. Then, for the inner optimization in~\eqref{eq:def12}, we show that the minimal value can be expressed as the sum of two generalized Rayleigh quotients, which depend on the variables $x$ and $\omega$. For this, we first solve two mapping problems.

    By Theorem~\ref{map:herm} and~\cite[Theorem~2.3]{MehMS16}, there exist a skew-Hermitian matrix $\Delta_J$ and a Hermitian negative semidefinite matrix $\Delta_R\le 0$ satisfying the constraints in~\eqref{eq:def12} if and only if
    \[
    x^*(i\omega E - J)x \in i\mathbb{R}
    \quad \text{and} \quad
    -x^*Rx<0.
    \]
    These conditions are automatically satisfied due to the structural properties of the matrices $E,~J$, and $R$.
    Among all such admissible mappings, the minimal norms are given by
    \begin{equation}\label{norm:jdef}
		\|\hat \Delta_J\|=\frac{\|(i\omega E - J)x\|}{\|x\|}, \quad \text{and} \quad
		\|\hat \Delta_R\|=\frac{\|Rx\|^2}{x^*Rx},
	\end{equation}
    and are attained by $\hat \Delta_J=-i \hat H_{(x, i(i\omega E-J)x)}$, where $\hat H$ is defined in~\eqref{def:hatH}, and $\hat \Delta_R=\frac{-1}{x^*Rx}(Rx)(Rx)^*$ (if $Rx\neq 0$). If $Rx=0$ (this case may arise if $R$ is singular), then again from {Theorem~\ref{map:def}}, $\hat \Delta_R=0$ is the solution that satisfies~\eqref{norm:jdef} by defining the fraction ${0}/{0}$ to have value $0$. 
    Note that $$ (\hat \Delta_J,\hat \Delta_R) \in \{(\Delta_J,\Delta_R): \Delta_J^*=-\Delta_J, \Delta_R^*=\Delta_R\leq 0\}\supseteq \mathcal S_d(J,R).$$
	%
	%
	%
    Hence, it is enough to show that $R+\hat \Delta_R\geq 0$. Then $(\hat \Delta_J,\hat \Delta_R)\in S_d(J, R)$, and since it is optimal over the larger set, it is also optimal over $S_d(J, R)$. Now in view of {\cite[Lemma 4.1]{MehMS16}} it follows that for the optimal perturbation $\hat \Delta_R$ we have $R+\hat \Delta_R\geq0$ and thus $(\hat \Delta_J,\hat \Delta_R)\in S_d(J, R)$.
    Using this in~\eqref{eq:def12}, we have
    \begin{equation}\label{eq:tempte1}
        (\vartheta^{S_d}_{\omega})^2=\inf_{x\in\C^{n}\setminus \{0\}}
    \frac{\|RQx\|^4}{(x^*QRQx)^2} + \frac{\|(i\omega I_n - JQ)x\|^2}{\|Qx\|^2}.
    \end{equation}
    Thus, the claim follows from~\eqref{eq:tempte1} and~\eqref{eq:def11}.

    Next we consider $\mathbb S=\mathcal S_i(J,R)$. By Theorem~\ref{map:herm}, there exist a skew-Hermitian matrix $\Delta_J$ and a Hermitian matrix $\Delta_R$ satisfying the constraints in~\eqref{eq:def12} if and only if
    \[
    x^* (i\omega E - J)x \in i\mathbb{R}
    \quad \text{and} \quad
    x^* R x \in \mathbb{R},
    \]
    which again holds trivially due to the structure of $E,~J$ and $R$.
    The corresponding minimal norms are
    \[
    \|\hat \Delta_J\| = \frac{\|(i\omega E - J)) x\|}{\|x\|},
    \qquad
    \|\hat \Delta_R\| = \frac{\|R x\|}{\|x\|},
    \]
    and are attained by
    \[
    \hat \Delta_J = -i\,\hat H_{(x, i(i\omega E - J) x)},
    \qquad
    \hat \Delta_R = \hat H_{(x,-Rx)},
    \]
    where $\hat H$ is defined in~\eqref{def:hatH}. Using minimal norm mappings and $S_i(J,R)\subseteq \{(\Delta_J, \Delta_R):\Delta_J^*=-\Delta_J, \Delta_R^*=\Delta_R\}$ in~\eqref{eq:def12}, we obtain
    \begin{align}\label{eq:idef1}
        (\vartheta^{S_i}_{\omega})^2&\geq \inf_{x\in \C^n\setminus\{0\}} \left\{ \frac{x^*(R^*R + (i\omega E - J)^*(i\omega E - J))x}{x^*Q^2x}\right\} \nonumber\\
        &= \sigma_{\text{min}}\left( \mat{c}R\\i\omega E - J \rix\right)^2.
    \end{align}
    This proves the inequality in~\eqref{eq:dist-imp1}. To prove the equality, we consider again two cases.

    Case 1: $R>0$. Let $\hat x$ be a right singular vector corresponding to the smallest singular value of $ \mat{c}R\\i\omega E - J \rix$. Using Lemma~\ref{map:herm}, there exists optimal
    $(\hat \Delta_J,\hat \Delta_R)$ such that $\hat \Delta_J^*=-\hat\Delta_J$, $\hat\Delta_R^*=\hat\Delta_R$ satisfying
       \begin{align}
       \hat\Delta_J\hat x&=(i\omega E-J)\hat x,  \quad
    \hat\Delta_R\hat x=-R\hat x, \label{norm:ideftemp1} \\
        \nnrm{(\hat \Delta_J,\hat \Delta_R)}^2&=\frac{\hat x^*(R^*R + (i\omega E - J)^*(i\omega E - J))\hat x}{\hat x^*\hat x}=\sigma_{\text{min}}\left( \mat{c}R\\i\omega E - J \rix\right).
        \label{norm:ideftemp2}
    \end{align}
    Thus, if we show that $R+\hat \Delta_R \geq 0$, then this will imply that $(\hat \Delta_J,\hat \Delta_R) \in S_i(J,R)$, and thus in view of~\eqref{norm:ideftemp2} and \eqref{eq:idef1}, we have equality in~\eqref{eq:dist-imp1}. Note that the optimal matrix $\hat H$ attaining the minimal norm in Theorem~\ref{map:herm} has at most one negative eigenvalue, since it is either a rank one or rank two matrix, and if it has rank two, then it is easy to check that $y\pm \frac{\|y\|}{\|x\|}x$ are eigenvectors corresponding to the eigenvalues $\pm \frac{\|y\|}{\|x\|}$, respectively. This implies that $\hat\Delta_R$ has at most one negative eigenvalue. By using~\eqref{norm:ideftemp1}, we obtain
    \begin{equation*}
        (R+\hat \Delta_R)\hat x= R\hat x + \hat \Delta_R \hat x= R\hat x - R\hat x=0.
    \end{equation*}
    Thus, $\hat\Delta_R$ is a matrix with at most one negative eigenvalue and also satisfies that $R+\Delta_R$ is singular. This implies from~\cite[Lemma 4.4]{MehMS16} that $ R+\hat \Delta_R \geq 0$.

    Case 2: $R\geq 0$ and singular. In this case, for any $\epsilon > 0$, $R_\epsilon=R+\epsilon I$ is positive definite and by following the arguments of Case 1, we have
    \[
    r^{S_i}(J,R_\epsilon)=\inf_{\omega\in \R} \vartheta^{S_i}_{\omega,\epsilon}= \inf_{\omega\in \R}\sigma_{\text{min}}\left( \mat{c}R_{\epsilon}\\i\omega E - J \rix\right),
    \]
    where $\vartheta^{S_i}_{\omega,\epsilon}$ is defined by~\eqref{eq:def12} when $R$ is replaced by $R_\epsilon$. Since singular values depend continuously on matrix entries, the function $\sigma_{\text{min}}\left( \mat{c}R_\epsilon\\i\omega E - J \rix\right)$
    is continuous in $\epsilon$ and $\omega$. Also for any
    $\epsilon >0$, the infimum of $\sigma_{\text{min}}\left( \mat{c}R_\epsilon\\i\omega E - J \rix\right)$ over $\omega$ is attained in a compact interval. This implies that
    $$r^{S_i}(J,R_\epsilon)=\inf_{\omega\in \R}\sigma_{\text{min}}\left( \mat{c}R_\epsilon\\i\omega E - J \rix\right)$$ is continuous with respect to $\epsilon$. Thus, we have
    \begin{eqnarray}
       r^{S_i}(J,R)=\lim_{\epsilon \to 0} r^{S_i}(J,R_\epsilon)=
       \lim_{\epsilon \to 0} \inf_{\omega\in \R}\sigma_{\text{min}}\left( \mat{c}R_\epsilon\\i\omega E - J \rix\right) = \inf_{\omega\in \R}\sigma_{\text{min}}\left( \mat{c}R\\i\omega E - J \rix\right).
    \end{eqnarray}
    This completes the proof.
\end{proof}
 In the following, we state a result for the structured distances to singularity $d_{sing}^{\mathcal S_d}(J,R)$ and $d_{sing}^{\mathcal S_i}(J,R)$ keeping $E$ fixed. We skip its proof as the distance $d_{sing}^{\mathcal S_d}(J,R)$ was obtained in~\cite[Table A.1]{PraS22} and is restated here adopting the convention that $\frac{0}{0}=0$. On the other hand, the proof for $d_{sing}^{\mathcal S_i}(J,R)$ follows by arguments similar to the proof of $d_{sing}^{\mathcal S_i}(E,J,R)$ in Theorem~\ref{thm:dist_singEJR_i}.

\begin{theorem}\label{thm:dist-singJR}
	Consider a robustly asymptotically stable dHDAE system of the form~\eqref{dhdae1}. Then, the distance to singularity $d_{sing}^{\mathbb S}(J,R)$ with respect to perturbations from the set $\mathbb S\in \{\mathcal S_d(J,R),\mathcal S_i(J,R)\}$ is $\infty$, if $E>0$. For singular $E\geq 0$, let the columns of the matrix $N$ form a basis of the null space of $E$. Then
    \begin{enumerate}
        \item the distance to singularity $d_{sing}^{\mathcal S_d}(J,R)$ with respect to perturbations from the set $\mathcal S_d$ is
    	\begin{equation}\label{eq:dist-sing}
    		(d_{sing}^{\mathcal S_d}(J,R))^2=
    			\inf_{x\in \C^{r}\setminus \{0\}} \left\{\left(\frac{x^*N^*R^2Nx}{x^*N^*RNx} \right)^2+\frac{x^*N^*J^*JNx}{x^*x} \right\},
    	\end{equation}
        \item the distance to singularity $d_{sing}^{\mathcal S_i}(J,R)$ with respect to perturbations from the set $\mathcal S_i(J,R)$ is
        \begin{equation}
            d_{sing}^{\mathcal S_i}(J,R) = \sqrt{\lambda_{\min}(N^*(R^2-J^2)N)}.
        \end{equation}
        %
    \end{enumerate}
\end{theorem}

 In the following theorem, we study the distance to higher index, allowing structure-preserving perturbations only in $J$ and $R$ while keeping $E$ fixed.

\begin{theorem}\label{thm:dist-hiJR}
    Consider an asymptotically stable dHDAE system of the form~\eqref{dhdae1}. Let $N$ be a matrix with orthonormal columns spanning a basis of $\ker(E)$. Then
    \begin{enumerate}
        \item the distance to higher index $d_{hi}^{\mathcal S_d}(J,R)$ with respect to perturbations from the set $\mathcal S_{d}(J,R)$ is given by
        \begin{equation}\label{eq:dist-hiJR1}
            (d_{hi}^{\mathcal S_d}(J,R))^2 \geq \inf_{x\neq 0} \left\{\frac{\|N^* J N x\|^2}{\|x\|^2} + \frac{\|N^* R N x\|^4} {\big(x^* N^* R N x\big)^2} \right\}.
        \end{equation}
        \sloppy Moreover, if the infimum in~\eqref{eq:dist-hiJR1} is attained at $\hat x \neq 0$ such that $R+\hat \Delta_R \geq 0$ with $\hat \Delta_R=-\frac{N\big(N^* R N \hat x\big) \big(N^* R N \hat x\big)^*N^*}{\hat x^* N^* R N \hat x}$, then we have equality in~\eqref{eq:dist-hiJR1},
        \item the distance to higher index $d_{hi}^{\mathcal S_i}(J,R)$ with respect to perturbations from the set $\mathcal S_{i}(J,R)$ is given by
        \begin{equation}\label{eq:dist-hiJR2}
            (d_{hi}^{\mathcal S_i}(J,R))^2 \geq \sigma_{\min}\left(\mat{c}N^*RN\\ N^*JN\rix \right). 
        \end{equation}
        Moreover, if $R+N\hat H_{(x,\,-N^* R N x)}N^* \geq 0$, where $x$ is the right singular vector corresponding to the smallest singular value of $\mat{c}N^*RN\\ N^*JN\rix$, then we have equality in~\eqref{eq:dist-hiJR2}.
    \end{enumerate}
\end{theorem}
\begin{proof}
    By Definition~\eqref{def:distances}, the distance to higher index under structured perturbations from the set $\mathbb S\in \{\mathcal S_d(J,R),\mathcal S_i(J,R)\}$ is given by
    \[
    d_{hi}^{\mathbb S}(J,R) = \inf\Big\{\nnrm{(\Delta_J,\Delta_R)} ~:~ (\Delta_J,\Delta_R)\in \mathbb S,~ \text{ index of } (E,J+\Delta_J-R-\Delta_R) > 1 \Big\}.
    \]
    Equivalently,
    \begin{equation*}
        d_{hi}^{\mathbb S}(J,R) = \inf\Big\{ \nnrm{(\Delta_J,\Delta_R)} ~:~ (\Delta_J,\Delta_R)\in \mathbb S,~ \text{rank}\!\big(N(E)^* (J+\Delta_J-R-\Delta_R) N(E) \big) < n \Big\},
    \end{equation*}
    where $N$ denotes an orthonormal basis matrix for the null space of $E$.
    Note that $d_{hi}^{\mathbb S}(J,R)=\mu_k^{\mathbb S}$, where $\mu_k^{\mathbb S}$ is defined by~\eqref{hu_hi} for $N_k=N$. Thus the proof follows by using the expressions of $\mu_k^{\mathcal S_d}$ and $\mu_k^{\mathcal S_i}$ from~\eqref{Sd:high11} and~\eqref{Sd:high22}, respectively.\end{proof}

\begin{table}[h!]\label{tablesummary}
\centering
\renewcommand{\arraystretch}{1.5}
\begin{tabular}{lccc}
\hline
\textbf{$\mathbb  S$}
& $\mathbf{d_{im}^{\mathbb S}}$
& $\mathbf{d_{sing}^{\mathbb S}}$
& $\mathbf{d_{hi}^{\mathbb S}}$ \\
\hline
$\mathcal S_d(E,J,R)$
& Lower bound: Theorem~\ref{thm:dist-imagEJR}
& Theorem~\ref{thm:dist-singEJR-d}
& Upper bound: Theorem~\ref{thm:dist-hiEJR} \\
$\mathcal S_i(E,J,R)$
& Lower bound: Theorem~\ref{thm:dist-imagEJR}
& Theorem~\ref{thm:dist_singEJR_i}
& Upper bound: Theorem~\ref{thm:dist-hiEJR} \\
$\mathcal S_d(J,R)$
& Theorem~\ref{thm:dist-imJR}
& Theorem~\ref{thm:dist-singJR}
& Lower bound: Theorem~\ref{thm:dist-hiJR} \\

$\mathcal S_i(J,R)$
& Theorem~\ref{thm:dist-imJR}
& Theorem~\ref{thm:dist-singJR}
& Lower bound: Theorem~\ref{thm:dist-hiJR} \\
\hline
\end{tabular}
\caption{Summary of structured distances and corresponding results}
\end{table}

In this section we have presented
bounds and analytic expressions for  the different structured distances. They are summarized in Table~\ref{tablesummary}. In the next section we illustrate these results via  numerical examples.

\section{Numerical examples}\label{sec:numerical}
In this section, we present numerical experiments illustrating the theoretical results developed in the previous sections. In particular, we compute the
structured distances to systems possessing purely imaginary eigenvalues, to
singular systems, and to systems of higher index for dissipative Hamiltonian
differential-algebraic equations (dHDAEs). All numerical computations are
carried out in MATLAB R2024b.

We briefly summarize below the numerical procedures used for the different
structured stability radii.

\begin{enumerate}
    \item[\rm (i)] \textbf{Perturbations in \(E\), \(J\), and \(R\).}
    We first consider the case in which structured perturbations are allowed in
    all three matrices \(E\), \(J\), and \(R\).

    \begin{itemize}
        \item For \(d_{im}^{\mathcal S_d}(E,J,R)\), the characterization leads
        to a constrained nonlinear optimization problem. We compute this
        quantity using the \texttt{fmincon} routine in MATLAB.
        Since the accuracy of \texttt{fmincon} may depend on the initial guess
        and on the conditioning of the underlying constrained problem, the
        resulting value should be regarded as a numerical estimate. Developing
        a numerical method that exploits the particular structure of this
        optimization problem is left for future work.

        \item For \(d_{im}^{\mathcal S_i}(E,J,R)\), the characterization
        involves an inner constrained optimization problem together with an
        outer minimization over the frequency parameter \(\omega\). The inner
        problem is solved using \textsc{CVX}, see \cite{GraB14,GraB08}, whereas the outer minimization
        with respect to \(\omega\) is performed using \texttt{fminsearch} in MATLAB.

        \item The distance \(d_{sing}^{\mathcal S_d}(E,J,R)\) is computed by
        solving the corresponding nonlinear optimization problem using
        \texttt{fminsearch}. In contrast,
        \(d_{sing}^{\mathcal S_i}(E,J,R)\) reduces to the computation of the
        smallest eigenvalue of an appropriate Hermitian matrix and is evaluated
        directly using the \texttt{eig} routine in MATLAB.

        \item For \(d_{hi}^{\mathcal S_d}(E,J,R)\), the characterization leads
        to an optimization problem involving generalized Rayleigh quotients.
        We use the nonlinear eigenvalue problem (NEPv) characterization and
        solve the associated NEPv using level-shifted self-consistent field
        (SCF) iterations; see~\cite{LuPSB25} for details. On the other hand,
        \(d_{hi}^{\mathcal S_i}(E,J,R)\) reduces to a standard eigenvalue
        problem and is computed using \texttt{eig}.
    \end{itemize}

    \item[\rm (ii)] \textbf{Perturbations in \(J\) and \(R\) only.}
    We next consider the partial perturbation setting in which \(E\) is kept
    fixed and structured perturbations are allowed only in \(J\) and \(R\).

    \begin{itemize}
        \item For \(d_{im}^{\mathcal S_d}(J,R)\), the inner optimization is
        formulated through an NEPv and solved using level-shifted SCF
        iterations. The resulting objective depends on the frequency parameter
        \(\omega\), and the outer minimization over \(\omega\) is carried out
        using \texttt{fminsearch}.

        \item The computation of \(d_{im}^{\mathcal S_i}(J,R)\) amounts to
        minimizing the smallest singular value of a parameter-dependent matrix
        over \(\omega\). For this problem, we employ a level-set technique.

        \item The distances \(d_{sing}^{\mathcal S_d}(J,R)\) and
        \(d_{hi}^{\mathcal S_d}(J,R)\) are computed from their corresponding
        NEPv formulations using level-shifted SCF iterations.

        \item Finally, \(d_{sing}^{\mathcal S_i}(J,R)\) and
        \(d_{hi}^{\mathcal S_i}(J,R)\) reduce to standard eigenvalue problems
        and are evaluated directly using \texttt{eig}.
    \end{itemize}
\end{enumerate}

The codes and data of the examples presented below are available at \url{https://gitlab.mpi-magdeburg.mpg.de/prajapati/dhdae-stability-radii.git}.




\begin{example}{\rm For the structure in the acoustic field problem in Example~\ref{ex:MDK1}, we performed a numerical experiment with randomly generated matrices $M, K$, and $D\in \C^{4,4}$ in MATLAB and enforced symmetry (and positive (semi-)definiteness where required). We constructed $M \in \mathbb{C}^{4\times 4} $ with rank $2$, so that the matrix $E$ becomes singular and, consequently, the system possesses two infinite eigenvalues. For the generated example,  using the criteria from Section~\ref{sec:distalg},  we verified that the resulting dHDAE system is robustly asymptotically stable. We then computed all the structured distances introduced in Section~\ref{sec:partial}; the resulting values are summarized in Table~\ref{tab:example1} below.
    \begin{table}[h!]
        \centering
        \renewcommand{\arraystretch}{1.5}
        \begin{tabular}{cccc}
            \hline
            \textbf{$\mathbb  S$}
            & $\mathbf{d_{im}^{\mathbb S}(J,R)}$
            & $\mathbf{d_{sing}^{\mathbb S}(J,R)}$
            & \textbf{l.b. to } $\mathbf{d_{hi}^{\mathbb S}(J,R)}$ \\

            & Theorem~\ref{thm:dist-imJR}
            & Theorem~\ref{thm:dist-singJR}
            & Theorem~\ref{thm:dist-hiJR}\\
            \hline
            $\mathcal S_d(J,R)$
            & 1.7631
            & 5.3705
            & 1.6546 \\

            $\mathcal S_i(J,R)$
            & 1.4351
            & 4.7242
            & 1.6546 \\
            \hline
        \end{tabular}
        \caption{Various structured distances.}
        \label{tab:example1}
    \end{table}
    Recall that the distance to the boundary of the set of robustly asymptotically stable systems is defined as the minimum of the three distances listed in Table~\ref{tab:example1}. Hence, $d_{inst}^{\mathcal S_d}(J,R)=1.6546$ and $d_{inst}^{\mathcal S_i}(J,R)= 1.4351$.

    Interestingly, the mechanism through which the boundary of the robustly asymptotically stable region is first reached depends on the chosen perturbation set. When perturbations are restricted to the structure-preserving set $\mathcal S_d(J, R)$, the  nearest  system that is not robustly asymptotically stable is obtained due to an increase in the  index of the system, while the spectral abscissa being negative and the regularity are preserved. In contrast, for the structure-preserving set $\mathcal S_i(J, R)$, the minimal distance for a loss of robust asymptotic stability occurs due to an eigenvalue reaching the imaginary axis.
    \begin{figure}[!h]
        \centering
        \includegraphics[width=1\textwidth]{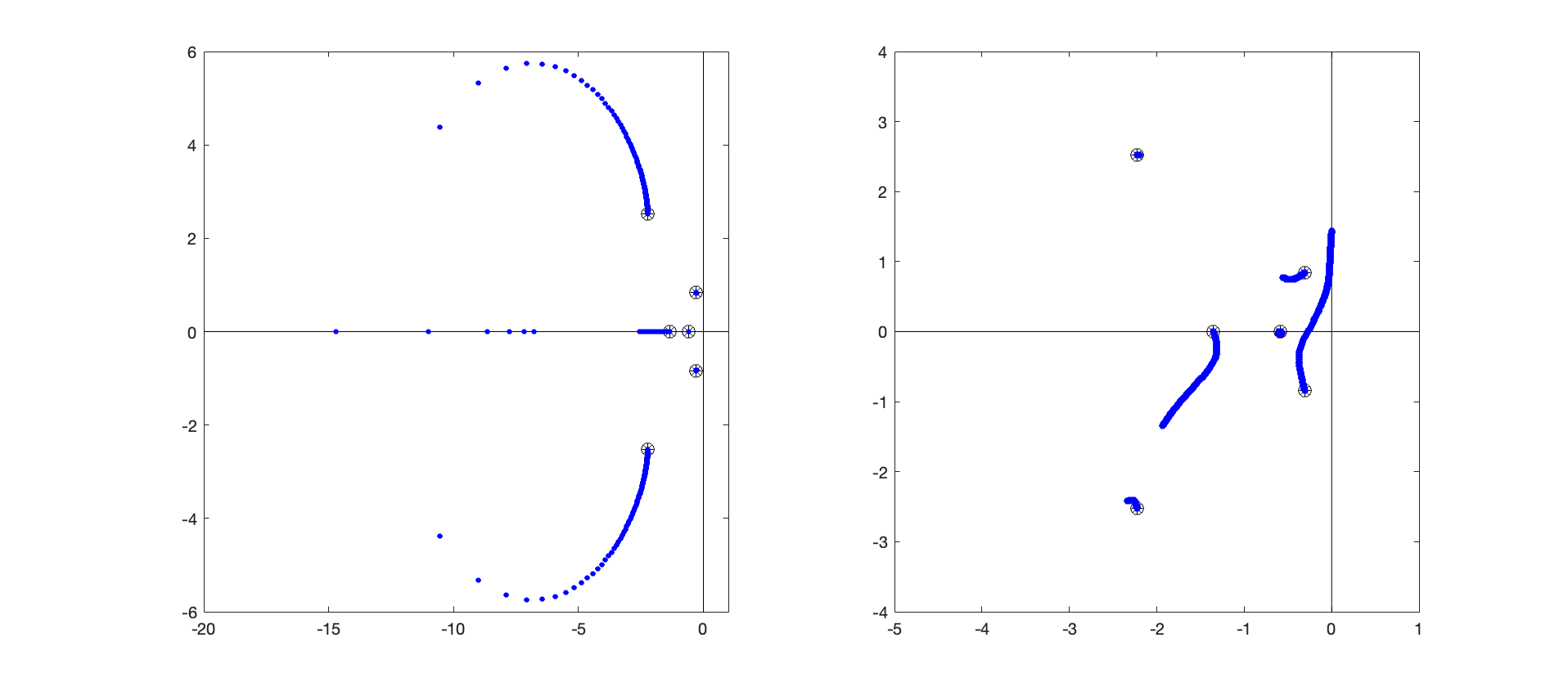}
        \caption{Eigenvalue perturbation curves of dHDAE system with respect to structured complex perturbations from the set $\mathcal S_d(J,R)$ (left) and $\mathcal S_i(J,R)$ (right).}
        \label{fig:example1}
    \end{figure}
    To illustrate the corresponding optimal perturbations $(\Delta_J,\Delta_R)$, we plot the eigenvalues of the pencil
    \[
    zE - (J - R)
    \]
    under the homotopy
    \[
    zE - (J + t\Delta_J - R - t\Delta_R),
    \qquad t \in [0,1],
    \]
where $\Delta_J$ and $\Delta_R$ are the perturbations that give the minimal distance.
Figure~\ref{fig:example1} shows the eigenvalue trajectories when the optimal perturbations are taken from $\mathcal S_d(J,R)$~(left) and $\mathcal S_i(J,R)$~(right). The original eigenvalues are marked by stars surrounded by circles, and as $t$ increases from $0$ to $1$, one observes that the loss of robustness occurs due to an increase in the index of the system (left) and due to an eigenvalue on the imaginary axis (right). It turns out that the perturbations that lead to a loss of robust asymptotic stability have the form $\Delta_J=0 $ and $\Delta_R=\diag(\Delta,0)$, so that the zero block structure of the original system is preserved. In the left plot of Figure~\ref{fig:example1}, one observes that a pair of finite eigenvalues approaches the real axis, coalesces, and then splits: one branch moves toward infinity while the other moves toward zero. The migration of a finite eigenvalue to infinity signals  an increase in the multiplicity of the infinite eigenvalue. This behavior is characteristic for a transition to a higher index. This also explains why the eigenvalue plot corresponding to $\mathcal S_d(J,R)$ exhibits symmetry about the real axis.


For the same constructed example, we additionally computed all three structured distances with respect to perturbations from the full structure-preserving sets $\mathcal S_d(E,J,R) \text{ and }\mathcal S_i(E,J,R)$, where all three matrices $E$, $J$, and $R$ are allowed to vary. The corresponding characterizations and bounds are presented in Theorems~\ref{thm:dist-imagEJR},~\ref{thm:dist_singEJR_i}, and~\ref{thm:dist-hiEJR}.
The computed values of these quantities are summarized in the Table~\ref{tab:example1a}.
    \begin{table}[h!]
        \centering
        \renewcommand{\arraystretch}{1.5}
        \begin{tabular}{cccc}
            \hline
            \textbf{$\mathbb  S$}
            & \textbf{l.b. to} $\mathbf{d_{im}^{\mathbb S}(E,J,R)}$
            & $\mathbf{d_{sing}^{\mathbb S}(E,J,R)}$
            & \textbf{u.b. to} $\mathbf{d_{hi}^{\mathbb S}(E,J,R)}$ \\

            & Theorem~\ref{thm:dist-imagEJR}
            & Theorem~\ref{thm:dist_singEJR_i}
            & Theorem~\ref{thm:dist-hiEJR}\\
            \hline

            $\mathcal S_d(E,J,R)$
            & 2.1984
            & 6.3032
            & 5.9784
            \\

            $\mathcal S_i(E,J,R)$
            & 1.1467
            & 2.3141
            & 1.6229 
            \\
            \hline
        \end{tabular}
        \caption{Various structured distances and bounds.}
        \label{tab:example1a}
    \end{table}
    %

}
\end{example}

\begin{example}{\rm
    Consider the DC power network example from~\cite{MehM19}. The system can be written as dHDAE
    \[
    E\dot{x} = (J - R)x,
    \]
    where the matrices $E, J, R \in \mathbb{R}^{5\times 5}$ are given by
    \[
    E = \text{diag}(L, C_1, C_2, 0, 0),\quad
    J = \begin{bmatrix}
    0 & -1 & 1 & 0 & 0 \\
    1 & 0 & 0 & -1 & 0 \\
    -1 & 0 & 0 & 0 & -1 \\
    0 & 1 & 0 & 0 & 0 \\
    0 & 0 & 1 & 0 & 0
    \end{bmatrix},\quad
    R = \begin{bmatrix}
    R_L & 0 & 0 & 0 & 0 \\
    0 & 0 & 0 & 0 & 0 \\
    0 & 0 & 0 & 0 & 0 \\
    0 & 0 & 0 & R_G & 0 \\
    0 & 0 & 0 & 0 & R_R
    \end{bmatrix},
    \]
    with inductance $L>0$, capacitances $C_1,C_2>0$, and resistances $R_G,R_L,R_R>0$. Choose the parameters \(L, C_1, C_2, R_L, R_G, R_R\) randomly using the \texttt{randn} command in MATLAB. Using the criteria from Section~\ref{sec:distalg},  we have verified that the resulting dHDAE system is robustly asymptotically stable. The generalized eigenvalues of the matrix pencil $(E, J-R)$ are $-37.1617, -2.7775 + 4.0447i, -2.7775 - 4.0447i, \infty, \infty$.  Hence, all finite eigenvalues lie strictly in the open left half-plane, while two infinite eigenvalues arise due to the algebraic constraints. The system is regular and of index one.

    We computed the various structured distances introduced in Section~\ref{sec:partial},  while perturbing only $J$ and $R$. The distances are summarized in Table~\ref{tab:example2}.
    \begin{table}[h!]
        \centering
        \renewcommand{\arraystretch}{1.5}
        \begin{tabular}{cccc}
            \hline
            \textbf{$\mathbb  S$}
            & $\mathbf{d_{im}^{\mathbb S}(J,R)}$
            & $\mathbf{d_{sing}^{\mathbb S}(J,R)}$
            & $\mathbf{d_{hi}^{\mathbb S}(J,R)}$ \\

            & Theorem~\ref{thm:dist-imJR}
            & Theorem~\ref{thm:dist-singJR}
            & Theorem~\ref{thm:dist-hiJR}\\
            \hline
            $\mathcal S_d(J,R)$
            & 1.1043
            & 1.6714
            & 1.3393     \\

            $\mathcal S_i(J,R)$
            & 0.9759
            & 1.6714
            & 1.3393 \\
            \hline
        \end{tabular}
        \caption{Various structured distances.}
        \label{tab:example2}
    \end{table}
%
We report the values of \(d_{hi}^{\mathcal S_d}(J,R)\) and \(d_{hi}^{\mathcal S_i}(J,R)\) in  Table~\ref{tab:example2} as exact distances. Indeed, for the corresponding optimal perturbations, we verified numerically that the additional condition $R+\Delta_R \ge 0$ is satisfied. Hence, the bound established in Theorem~\ref{thm:dist-hiJR} is attained in this example and therefore holds with equality.

Hence, the distance to the boundary of the region of robust asymptotic stability with respect to the perturbation set $\mathcal S_d(J,R)$ is $d_{inst}^{\mathcal S_d}(J,R)=1.1043$, while with respect to $\mathcal S_i(J,R)$ it is $d_{inst}^{\mathcal S_d}(J,R)=0.9759$. \\

The computed  structured distances respectively bounds with respect to perturbations from the full structure-preserving sets $\mathcal S_d(E,J,R) \text{ and }\mathcal S_i(E,J,R)$ 
are summarized in Table~\ref{tab:example2a}.
    \begin{table}[h!]
        \centering
        \renewcommand{\arraystretch}{1.5}
        \begin{tabular}{cccc}
            \hline
            \textbf{$\mathbb  S$}
            & \textbf{l.b. to} $\mathbf{d_{im}^{\mathbb S}(E,J,R)}$
            & $\mathbf{d_{sing}^{\mathbb S}(E,J,R)}$
            & \textbf{u.b. to} $\mathbf{d_{hi}^{\mathbb S}(E,J,R)}$ \\

            & Theorem~\ref{thm:dist-imagEJR}
            & Theorem~\ref{thm:dist_singEJR_i}
            & Theorem~\ref{thm:dist-hiEJR}\\
            \hline

            $\mathcal S_d(E,J,R)$
            & 1.1769
            & 1.1061
            & 1.1573
            \\

            $\mathcal S_i(E,J,R)$
            & 0.0192
            & 1.0015
            & 1.0002
            \\
            \hline
        \end{tabular}
        \caption{Various structured distances and bounds.}
        \label{tab:example2a}
    \end{table}
Each of the three distances
provides an upper bound for the distance to nearest system that is not robust asymptotically stable. }
\end{example}

\section{Conclusions}\label{sec:conclusion}

For semi-dissipative Hamiltonian differential-algebraic systems and the associated matrix pencils, we have presented explicit characterizations as well as bounds for the distance to the nearest system with purely imaginary eigenvalues, the nearest system of index higher than one and the nearest  singular system. Questions of further research are  the equality cases when we only have lower and upper bounds, as well as  appropriate numerical optimization methods to compute these distances. The extension to large-scale systems will require a combination with model reduction techniques and will be the topic of future work.


\section*{Author contributions}%
\addcontentsline{toc}{section}{Author contributions}
 Peter Benner performed writing - review and editing, and funding acquisition. Volker Mehrmann and Punit Sharma performed writing - review and editing, and conceptualization.
 Anshul Prajapati performed writing - original draft, conceptualization and numerical implementation.


\section*{Acknowledgments}%
\addcontentsline{toc}{section}{Acknowledgments}
Punit Sharma acknowledges the support of the SERB-CRG grant (CRG/2023/003221) by Government of India.

\addcontentsline{toc}{section}{References}

\end{document}